\documentclass[11pt,reqno]{amsart}
\usepackage[T1]{fontenc}
\usepackage{amsmath,esint}
\usepackage{amsthm}
\usepackage{amsfonts}
\usepackage{enumerate}

\usepackage{arydshln}
\usepackage{amssymb}
\usepackage{a4wide}

\newcommand{\bea}{\begin{eqnarray}}
\newcommand{\eea}{\end{eqnarray}}	
\newcommand{\bee}{\begin{eqnarray*}}
\newcommand{\eee}{\end{eqnarray*}}
\newcommand\NN{\mathbb{N}}

\newcommand\RR{\mathbb{R}}

\newcommand\OO{\mathcal{O}}

\newcommand\EE{\mathcal{E}}
\newcommand\WW{\mathbf{W}}

\newcommand\e{\epsilon}
\newcommand\be{\begin{equation}}
\newcommand\ee{\end{equation}}

\newcommand\la{\langle}
\newcommand\ra{\rangle}

\newcommand\p{\partial}

\newcommand\im{\mathrm{\, Im\,}}

\newcommand\R{\RR}

\newtheorem*{theorem*}{Main Theorem} 
\newtheorem{theorem}{Theorem}
\newtheorem{lemma}[theorem]{Lemma}

\newtheorem{proposition}[theorem]{Proposition}
\newtheorem{remark}{Remark}
\numberwithin{equation}{section}

\title[Multi-solitons with logarithmic distance]{Strongly interacting multi-solitons with logarithmic relative distance for gKdV equation}
\author[T.V. Nguy\~{\^e}n]{Nguy\~{\^e}n Ti\'{\^e}n Vinh}
\address{CMLS, \'{E}cole Polytechnique, CNRS, Universit\'{e} Paris-Saclay, 91128 Palaiseau, France}
\email{tien-vinh.nguyen@polytechnique.edu}
\begin{document}
\begin{abstract}
We consider the following class of equations of (gKdV) type
$$\p_t u + \p_x (\p_x^2 u + |u|^{p-1}u) = 0, \quad p\mbox{ integer},\quad t,x \in \R$$
with mass sub-critical ($2< p<5$) and mass super-critical nonlinearities ($p> 5$). We prove the existence of 2-solitary wave solutions with logarithmic relative distance, i.e. solutions $u(t)$ satisfying
\[\left\|u(t)- \bigg( Q (\cdot - t - \log (ct)) + \sigma Q (\cdot -  t + \log (ct))\bigg)\right\|_{H^1}\to 0 \ \ \mbox{as} \ \ t\to +\infty,\]
where $c=c(p)> 0$ is a fixed constant, $\sigma = -1$ in sub-critical cases and $\sigma = 1$ in super-critical cases. For the integrable case ($p=3$), such solution was known by integrability theory. This regime corresponds to strong attractive interactions. For sub-critical $p$, it was known that opposite sign traveling waves are attractive. For super-critical $p$, we derive from our computations that same sign traveling waves are attractive. 
\end{abstract}
\maketitle
\section{Introduction}
We consider the following class of equations of (gKdV) type
\be\label{gkdv}
\begin{cases}
\p_t u + \p_x(\p_x^2 u + |u|^{p-1}u) = 0 ,\qquad t,x \in \RR\\
u(0,x) = u_0 (x), \qquad u_0 \in H^1 : \RR \to \RR
\end{cases}
\ee
where $p > 2$ is an integer. When $p$ is odd $(p=3,5,7,...)$, \eqref{gkdv} is the usual generalized Korteweg-de Vries equations (gKdV). In particular, the case $p=3$ corresponds to the modified Korteweg-de Vries (mKdV) equation, which is a completely integrable model \cite{PCS}.

From \cite{KPV}, the Cauchy problem for \eqref{gkdv} is well-posed: for all $u_0 \in H^1(\RR)$, there exists $T^* > 0$ and a solution $u \in \mathcal{C} ([0, T^*], H^1(\RR))$ to \eqref{gkdv}, unique in some class $Y_{T^*} \subset \mathcal{C} ([0, T^*), H^1(\RR))$. Moreover, the following blow up criterion holds:
\[T^* < + \infty \,\,\mbox{  implies  } \lim_{t \to T^*} \| u(t)\|_{H^1} = + \infty.\]
Recall that such $H^1$ solutions $u(t)$ satisfy the conservation of mass and energy:
\begin{equation*}
\begin{gathered}
M(u(t)) = \int u^2(t,x) dx = M_0, \\\quad E(u(t)) = \frac{1}{2} \int (\p_x u)^2(t,x) dx - \frac{1}{p+1} \int |u|^{p+1}(t,x) dx =E_0.
\end{gathered}
\end{equation*}
For $p<5$, all $H^1$ solutions are global in time, as a consequence of the conservation laws and the Gagliardo-Nirenberg inequality
\[\|u\|_{L^{p+1}} \lesssim \|u\|_{L^2}^{1-\sigma} \|\p_x u\|^\sigma_{L^2}\quad \mbox{with} \quad \sigma  = \frac{p-1}{2(p+1)}.\]
For $p=5$, called the critical case, blow-up in finite time is possible (see e.g. \cite{MMR1}).

A very important feature of these equations is the existence of traveling wave solutions, usually called solitons, of the form
\[R_{v_0,x_0} = Q_{v_0} (x-x_0 - v_0t)\]
where $Q_v = v^{\frac{1}{p-1}} Q (\sqrt{v} x)$ and $Q$ is the ground state solitary wave, i.e. the unique positive even solution of the equation
\[Q'' + Q^p = Q, \qquad Q (x) = \left(\frac{p+1}{2 \cosh^2\left( \frac{p-1}{2}x\right)} \right)^{\frac{1}{p-1}}.\]
Recall that for sub-critical cases ($p<5$), the solitons are stable and asymptotically stable in $H^1$ in some sense (see \cite{CL}, \cite{MMarch}, \cite{PW}, \cite{We86}) while for critical and super-critical cases ($p\geq 5$), the solitons are unstable (see \cite{VC1}, \cite{GSS}, \cite{MMR1}, \cite{PW2}). 
\subsection{Main results} In this article, we construct a 2-solitary wave solution with logarithmic relative distance.

\begin{theorem*}[Multi-solitons with logarithmic distance]
\label{main:th} Let $p$ integer, $p \neq 5$ and $p>2$. There exist $t_0 >0$ and an $H^1$ solution on $[t_0, + \infty)$ of~\eqref{gkdv}
which decomposes asymptotically into two solitary waves 
\begin{equation}\label{main:eq}
\left\|u(t)- \bigg( Q (\cdot - t - \log (ct)) + \sigma Q (\cdot -  t + \log (ct))\bigg)\right\|_{H^1}\to 0 \ \ \mbox{as} \ \ t\to +\infty,
\end{equation}
where
\begin{equation}\label{c:eq}
c= c(p)=  \sqrt{\frac{ 8(p-1)}{|5-p|}} \left(2p+2\right)^{\frac{1}{p-1}} \|Q\|_{L^2}^{-1} > 0 ,
\end{equation} 
$\sigma = -1$ in sub-critical cases ($2< p<5$) and $\sigma = 1$ in super-critical cases ($p>5$).
\end{theorem*}

In sub-critical cases, \cite{Mi03} proved that the interaction of two solitons of same sign is repulsive. The regime displayed in Main Theorem corresponds to attractive interaction between solitons and thus $\sigma = -1$ for $p<5$. For the integrable case ($p = 3$), the existence of "double pole solutions", two solitons with alternative sign corresponding to the regime in Main Theorem, was reported in \cite{OW} by using inverse scattering transform (see also Remark \ref{rk3}). 

In super-critical cases, we derive from our computations that the interaction between two solitons with the same sign is attractive which explains that $\sigma = 1$ for $p>5$. In particular, we can apply the strategy of this article to construct multi-solitary waves with logarithmic relative distance for super-critical (gKdV) $\p_t u + \p_x(\p_x^2 u + u^p) = 0$ with $p$ even ($p=6,8,...$).

We observe that the relative distance of the two traveling waves is $2 \log (ct)$ with $c$ given in \eqref{c:eq}. We expect that this is the unique regime of logarithmic relative distance for this scaling.

\smallskip
We point out similarity with the result proved by the author in \cite{NV1} for nonlinear Schr\"odinger equations $i \p_t u + \Delta u + |u|^{p-1}u = 0$: for any dimension $d \geq 1$ and any $H^1$ sub-critical nonlinearity $p$, except the $L^2$ critical one $p = 1+\frac{4}{d}$, there exists an $H^1$ solution $u(t)$ which decomposes asymptotically
\[\left\|u(t)- e^{i\gamma(t)}\sum_{k=1}^2 Q(.-x_k(t))\right\|_{H^1 (\R^d)} \lesssim \frac{1}{t}
\]
where $x_1(t) = - x_2(t)$ and $|x_1(t) - x_2(t)| = 2(1 + o(1))\log t\ \mbox{as} \ t\to +\infty.$

\begin{remark} Our result holds in both mass sub-critical ($p<5$) and mass super-critical cases ($p>5$). For the mass critical case $p=5$, we conjecture that solution such as in Main Theorem does not exist. Indeed, as for nonlinear Schr\"odinger equations, the instability directions related to scaling should be excited by the nonlinear interactions (see \cite{MRlog}, \cite{NV1}). 

By scaling and translation, for any $v >0$ and $x_0\in \R$, there exists an $H^1$ solution such that as $t\to +\infty$
\[\left\|u(t)- \bigg( Q_v (\cdot - x_0- vt - \frac{1}{\sqrt{v}}\log (cv^{\frac{3}{2}}t)) + \sigma Q_v (\cdot - x_0 -  vt + \frac{1}{\sqrt{v}}\log (cv^{\frac{3}{2}}t))\bigg)\right\|_{H^1}\to 0.\]
It is mandatory that the scaling $v$ is the same for both solitons since otherwise they would have different speeds, see next remark.
\end{remark}
\begin{remark} In our main result, the logarithmic distance is due to strong attractive interaction between the two solitary waves. This is in contrast with most previous works on multi-solitary waves of (gKdV) where weak interactions do not change the behavior of solitons, see in particular \cite{CMM}, \cite{VC2}, \cite{Martel1}, \cite{MMT2}. A typical example to illustrate weakly interacting dynamics is the existence of multi-soliton solutions $u(t)$ of (gKdV) with any different speeds $0 < v_1 <...<v_K$ and any $x_1,...,x_K \in \R$,
\[\left\|u(t) - \sum_{k=1}^K Q_{v_k} (\cdot - v_k t -x_k)\right\|_{H^1} \to 0, \qquad \mbox{ as } t \to +\infty.\]
Combining the construction of this paper and the construction of multi-soliton solutions with weak interactions in \cite{CMM}, \cite{Martel1}, we prove the existence of $u(t)$ such that as $t \to + \infty$
\begin{multline*}
\bigg\|u(t)- \bigg[ \sum_{k=1}^K \bigg( Q_{v_k} (\cdot - v_k t - \frac{1}{\sqrt{v_k}}\log (c v_k^{\frac{3}{2}} t)) + \sigma Q_{v_k} (\cdot -  v_k t + \frac{1}{\sqrt{v_k}}\log (c v_k^{\frac{3}{2}} t)) \bigg)\\
+ \sum_{k'=1}^{K'} Q_{w_{k'}} (\cdot - w_{k'} t) \bigg]\bigg\|_{H^1}\to 0
\end{multline*}
for any $v_k, w_{k'} \neq 1$, $0<v_1 <...<v_K$, $0<w_1 <...<w_{K'}$ and $v_k \neq w_{k'}$. Since configurations of 2-soliton with logarithmic distance like in \eqref{main:eq} are determined by its strong interaction, which will not be affected by weak interactions.
\end{remark}
\begin{remark}\label{rk3} It is informative to observe the asymptotic form of the double pole solution in \cite{OW} as $t \to \pm \infty$ to remark the perfect interaction of the solitons in integrable case.  For $p=3$, the ground state solitary wave is $Q (x) = \sqrt{2} \cosh^{-1} (x) $ and the behavior of double pole solution at $t \to \pm \infty$ writes
\begin{align*}
u(x,t) &\sim \, Q (x - t + \log (4t)) - Q (x - t - \log (4t)), \qquad t \to + \infty\\
&\sim -\, Q (x - t + \log (-4t)) + Q (x - t - \log (-4t)), \qquad t \to - \infty
\end{align*}
(see the formula $\mathrm{3.13}$ in \cite{OW} with $\eta = 1$ after suitable scaling, note that $c = 4$ for $p=3$ matches \eqref{c:eq}) so soliton and antisoliton approach very slowly, interact nonlinearly and separate again very slowly. The distance between soliton and antisoliton is asymptotically proportional to $\log |t|$ both at $t \to + \infty$ and $t \to - \infty$. An interesting question is to understand the behavior of solutions in Main Theorem for $t \leq 0$ in non-integrable case ($p=4$). We conjecture that the behavior as $t \to - \infty$ for $p=4$ is not the same, the relative distance being of order $|t|$. A hint for this observation comes from computations in \cite{MaMeinvent}: when the dispersion is nontrivial, the faster soliton becomes bigger and the slower becomes smaller and then they should split linearly in time for $t \to - \infty$, in contrast with the integrable case. 
\end{remark}

We summarize the organization of the paper. In Section \ref{s:2}, we construct an approximate solution (an ansatz solution) and find the main order of all terms in the formal evolution system of the geometrical parameters (scaling, position). In Section \ref{sec:uni}, we prove backward uniform estimates, note that the proof of these estimates is slightly different in super-critical cases due to unstable directions (see \cite{CMM}). In Section \ref{s:3}, we finish the proof of Main Theorem by compactness arguments on a suitable sequence of backward solutions.
\subsection{Notation and identities on solitons} The $L^2$ scalar product of two real valued functions $f,g \in L^2 (\RR)$ is denoted by
\[ \la f,g \ra = \int_{\RR} f(x) g(x) dx.  \]
Recall the equation of $Q_v$
\begin{equation}\label{eq:Q}
Q_v = v^{\frac{1}{p-1}} Q (\sqrt{v} x), \qquad Q_v'' + Q^p_v = v Q_v 
\end{equation}
for $v>0$ where 
\[Q (x) = \left(\frac{p+1}{2\cosh^2\left(\frac{p-1}{2}x\right)} \right)^{\frac{1}{p-1}}\]
solves
\begin{equation}
Q'' + Q^p = Q \quad\mbox{and} \quad (Q')^2+ \frac{2}{p+1} Q^{p+1} = Q^2.
\end{equation}
It is easily checked that for $p>2$, as $x \to + \infty$
\begin{equation}\label{asyQ}
\begin{aligned}
Q (x) &=  c_Q e^{-x} + O(e^{-2x}),\quad Q' (x) &=  -c_Q e^{-x} + O(e^{-2x}).
\end{aligned}
\end{equation}
with $c_Q = \left(2p+2\right)^{\frac{1}{p-1}} $. Let
\[P= \frac{Q'}{Q} + 1 - \frac{2}{c_Q} e^{-x}Q,\]
then $|P(x)| \lesssim e^{- 2|x|}$.

We denote by $\mathcal{Y}$ the set of smooth functions $f$ satisfying
\[\mbox{for all } p \in \NN, \mbox{ there exists }q\in \NN,\mbox{ such that }\forall x\in \RR \quad |f^{(p)}| \lesssim |x|^q e^{-x}. \]
Let $\Lambda$ be the generator of $L^2$-scaling
\begin{equation}\label{def1}
\Lambda Q_v = \left(\frac{\p}{\p v'} Q_{v'}\right)_{|v' = v} = \frac{1}{v}\left(\frac{1}{p-1} Q_v + \frac{1}{2} xQ_v' \right),
\end{equation}
\begin{equation}\label{def2}
\Lambda^2 Q_v = \left(\frac{\p^2}{\p v'^2} Q_{v'}\right)_{|v' = v}.
\end{equation}
The linearization of \eqref{gkdv} involves the following self-adjoint operator
\[L f = - f'' + f - p Q^{p-1}f. \]
We recall the coercivity property of $L$ (see \cite{MMT2}, \cite{We86}) in sub-critical cases: there exists $\mu > 0$ such that for $f \in H^1(\RR)$,
\begin{equation}\label{coer:eq}
\la Lf,f \ra  \geq \mu \|f\|_{H^1}^2 - \frac{1}{\mu} \left(\la f,Q \ra^2 + \la f, Q'\ra^2 \right).
\end{equation}
The situation is different in super-critical cases since the direction related to the eigenvector $Q^{\frac{p+1}{2}}$ cannot be controlled by the scaling parameter. This is due to the unstable nature of the soliton and to the existence of eigenfunctions $Z^{\pm}$ with real nonzero eigenvalues of the operator $L\p_x$:
\begin{equation}\label{eqZ}
L(\p_x Z^{\pm}) = \pm e_0 Z^{\pm}, \qquad e_0>0
\end{equation}
constructed in \cite{L}, \cite{PW}. The functions $Z^{\pm}$ are normalized so that $\|Z^\pm\|_{L^2} =1$. We recall from \cite{PW} that $Z^\pm \in \mathcal{Y}$ and from \cite{CMM} (see also \cite{DM}) that there holds a property of positivity based on $Z^\pm$: 
there exists $\mu > 0$ such that for $f \in H^1(\RR)$,
\begin{equation}\label{coer:eq1}
\la Lf,f \ra  \geq \mu \|f\|_{H^1}^2 - \frac{1}{\mu} \left(\la f,Z^+ \ra^2+\la f,Z^- \ra^2 + \la f, Q'\ra^2 \right).
\end{equation}
Now, we give here some explicit antecedents and integral identities for $L$:
\begin{equation}\label{Q:iden}
L Q = - (p-1)Q^p, \qquad LQ' =0, \qquad\left(\frac{Q'}{Q} \right)' = - \frac{p-1}{p+1}Q^{p-1},
\end{equation}
\begin{equation}\label{Q:iden2}
L\left(\frac{Q'}{Q} \right) =  -\frac{3p-1}{p+1}Q^{p-2}Q'+ \frac{Q'}{Q},\quad \left(L\left(\frac{Q'}{Q} \right)\right)'= - \frac{3p(p-1)}{p+1}Q^{p-1} + \frac{3 (3p-1)(p-1)}{(p+1)^2}Q^{2p-2},
\end{equation}
\begin{equation}\label{int:value}
\int Q \Lambda Q = \left(\frac{1}{p-1} - \frac{1}{4} \right) \int Q^2, \quad \int Q^{r +p-1} = \frac{r(p+1)}{2r + p-1}\int Q^r \mbox{ for } r \geq 1.
\end{equation}
We introduce here some notation of order
\begin{itemize}
\item $f(t,x) = O (g(t,x))$ if $\exists C > 0$ such that
\begin{equation*}
|f(t,x)| \leq C g(t,x) \quad \mbox{and} \quad |\p_x f(t,x)| \leq C g(t,x)
\end{equation*}
\item $f(t,x) = \OO (g(t,x))$ if $\exists C > 0$ such that
\begin{equation}
\left(1 + e^{\frac{1}{2}(x - z_1(t))}\right)|f(t,x)| \leq C g(t,x) \quad \mbox{and} \quad \left(1 + e^{\frac{1}{2}(x - z_1(t))}\right)|\p_x f(t,x)| \leq C g(t,x)
\end{equation}
\item $f(t,x) = O_{H^{1}} (g(t))$ if $\exists C > 0$ such that
\begin{equation*}
f(t,x) = O (g(t)) \quad \mbox{and} \quad\|f(t,x)\|_{H^1} \leq C g(t).
\end{equation*}
\end{itemize}

\section{Approximate solution}\label{s:2}
In this section, we first construct an almost symmetric 2-bubble approximate solution to renormalized equations of \eqref{gkdv} and then extract the evolution system of the geometrical parameters of the bubbles. The approximate solution contains special terms due to the nonlinear interactions of the waves (see Lemma \ref{lem:A}) which appear at the main order of the evolution system (see \eqref{flow}). This tail of order $e^{-z}$ is indeed relevant in the description of the exact solution, see Remark \ref{remark5}. We also state a standard modulation lemma around the approximate solution. 
\subsection{System of modulation equations}
We renormalize the flow by considering
\be\label{renor}
u (t,x) =  w(t,y),  \quad x = y + t
\ee
so that $w(t,y)$ verifies the equation
\begin{equation}\label{rescale:eq}
\p_t w + \p_y (\p^2_y w  - w + |w|^{p-1}w ) = 0.
\end{equation}
Consider a time dependent $\mathcal{C}^1$ function of parameters $\Gamma (t)$ of the form
\[\Gamma (t) = (\mu_1 (t), \mu_2(t), z_1(t), z_2(t)) \in \RR^+ \times \RR^- \times \RR \times \RR \]
with $|\mu_1|, |\mu_2| \ll 1, z(t) \gg 1$ and $|z_1(t) + z_2(t)| \ll 1$ where we denote
\[z (t) = z_1 (t) - z_2 (t).\] 
We look for an approximate solution to the problem. By expanding the first order of the interaction of the two solitons which is of order $e^{-z}$, we guess an anszat $r(t,y)$ for this order and deduce from the computations the evolution system of the geometrical parameters $\Gamma (t)$. Since the extra term $r(t,y)$ due to the interactions may not be in $H^1$ (it may have nonzero limits at $-\infty$), we have to introduce an $L^2$ approximation of these terms, using suitable cut-off functions. Note that in the integrable case ($p=3$), one should have $r(t,y)$ in $L^2$ (see Remark \ref{remark4}), thus the phenomenon is related to nonintegrability (see \cite{MaMeinvent} for a similar phenomenon).

Let $\psi: \RR \to [0,1]$ be a $\mathcal{C}^\infty$ function such that
\[\psi \equiv 0 \mbox{ on } ( - \infty, 0],\quad \psi \equiv 1 \mbox{ on } \bigg[\frac{1}{2}, + \infty\bigg) , \quad \psi' \geq 0\]
and 
\[\varphi (t, y) = \psi\left( e^{- \frac{1}{2}z(t)} y+ 1\right),\quad \tilde{\varphi} (t, y) = \psi'\left( e^{- \frac{1}{2}z(t)} y+ 1\right).\]
Then remark that 
\begin{equation}\label{remark:varphi}
\|\varphi (t,y)/ (1 + e^{\frac{1}{2}(y - z_1(t))}) \|_{L^2} + \|\tilde{\varphi} (t,y)/ (1 + e^{\frac{1}{2}(y - z_1(t))}) \|_{L^2}\lesssim e^{\frac{1}{4}z(t)},
\end{equation}
\begin{equation}\label{remark:varphi1}
\|\p_y^k \varphi (t,y)\|_{L^2} \lesssim  e^{\frac{1}{4}z(t)-\frac{k}{2}z(t)},\quad \mbox{for } k \in \mathbb{N}
\end{equation}
and
\begin{equation}\label{remark:varphi2}
\left|\frac{\p\varphi (t,y)}{\p z_k}\right| = \left|\frac{1}{2}e^{-\frac{1}{2}z(t)} y \psi' \left( e^{- \frac{1}{2}z(t)} y+ 1\right)\right| \lesssim \left|\tilde{\varphi} (t,y)\right|, \quad \mbox{for } k = 1,2
\end{equation}
as $\psi' \left( e^{- \frac{1}{2}z(t)} y+ 1\right) \equiv 0$ for $|y| \geq e^{\frac{1}{2}z(t)}$. Next, we set
\begin{align*}
\tilde{R}_k (t,y) = Q_{1 + \mu_k (t)} (y - z_k(t)),& \quad  R_k (t,y) = Q (y - z_k(t))\\
\Lambda\tilde{R}_k (t,y) =\Lambda Q_{1 + \mu_k (t)} (y - z_k(t)),& \quad  \Lambda R_k (t,y) =\Lambda Q (y - z_k(t))
\end{align*}   
and similarly for $\Lambda^2 R_k$, where $\Lambda Q_v$ and $\Lambda^2 Q_v$ are defined in \eqref{def1}, \eqref{def2}. Denote $\tilde{R} = \tilde{R}_1 + \sigma\tilde{R}_2$, let consider the approximate solution of the form
\begin{equation}\label{V0:eq}
V = \tilde{R}_1 + \sigma \tilde{R}_2 + \tilde{r} = \tilde{R}+ \tilde{r},\quad
\mbox{ where } \tilde{r} (t,y) = r (t,y) \varphi (t,y),
\end{equation}
with $r(t,y)$ to be determined.

\begin{proposition}[Approximate solution and leading order flow] \label{prop:lead} Let $I$ be some interval and a function of parameters $\Gamma (t)$ on $I$ such that
\begin{equation}\label{cond:G}
|\mu_1 (t) + \mu_2 (t)| \leq e^{- \frac{9}{16} z(t)}, \quad |z_1 (t) + z_2 (t)| \leq e^{- \frac{1}{32} z(t)},\quad z_1(t) - z_2(t) \geq 0.
\end{equation}
Then there exist unique real-valued functions $A_1 (y), A_2 (y)$ and some constants $\alpha > 0, \theta , a_1, a_2$ satisfying:
\begin{enumerate}[(a)]
\item $V (t,y)$ defined as in \eqref{V0:eq} with $r (t,y) = e^{- z(t)} [A_1(y - z_1(t)) + A_2 (y - z_2(t))]$
\begin{equation}
\begin{aligned}
V (t, y) =&  V(y;\Gamma (t))\\
=& \,Q _{1 + \mu_1(t)} (y- z_1(t)) + \sigma Q _{1 + \mu_2(t)} (y- z_2(t))\\
&+ e^{- z(t)}\left(A_1(y - z_1(t)) + A_2 (y - z_2(t))\right)\varphi (t,y)
\end{aligned}
\end{equation}
is an approximate solution of equation \eqref{rescale:eq} in the following sense: the error to the flow at $V$
\begin{equation}
\EE_{V}= \p_t V + \p_y (\p^2_y V - V +|V|^{p-1}V)
\end{equation}
decomposes as
\begin{equation}\label{flow}
\EE_V = \vec{m}_1 \cdot \vec{M}_1 +  \vec{m}_2 \cdot \vec{M}_2 + E(t,y)
\end{equation}
where
\begin{align}
\vec{m}_1 = \begin{pmatrix}
\dot{\mu}_1 +\alpha e^{-z}\\ \dot{z}_1 - \mu_1 - a_1 e^{-z}
\end{pmatrix}&, \quad \vec{m}_2 = \begin{pmatrix}
\dot{\mu}_2 -\alpha e^{-z}\\ \dot{z}_2 - \mu_2 + a_2 e^{-z}
\end{pmatrix},\\ \vec{M}_1 = \begin{pmatrix}
\Lambda \tilde{R}_1\\ -\p_y \tilde{R}_1 
\end{pmatrix}&,\quad \vec{M}_2 = \sigma\begin{pmatrix}
\Lambda \tilde{R}_2\\ - \p_y \tilde{R}_2
\end{pmatrix}
\end{align}
and 
\begin{multline}\label{e:R}
\|E (t,y) \|_{H^1} \lesssim |\mu_1| z e^{-\frac{3}{4} z} + |\mu_2| z e^{- \frac{3}{4}z}+ e^{-\frac{5}{4} z} + \sum_{j=1}^2 \left|\vec{m}_j \right| \left( e^{-z} + |\mu_1| e^{-\frac{3}{4}z} + |\mu_2| e^{-\frac{3}{4}z}\right).
\end{multline}
\item Closeness to the sum of two solitons. For some $q>0$,
\begin{equation}\label{closeH:V}
\left\|V(t) - \{Q_{1 + \mu_1(t)} (\cdot - z_1(t)) + \sigma Q_{1+\mu_2(t)}(\cdot - z_2(t))\}\right\|_{H^1} \lesssim |z(t)|^q e^{-\frac{3}{4} z(t)},\quad t\in I.
\end{equation} 
\item $A_j \in L^\infty (\RR), j=1,2$ are the unique solutions of:
\begin{equation}
\begin{gathered}
(-LA_1)' - \alpha \Lambda Q - a_1 Q' =  - \sigma p c_Q \, \p_x (e^{-x}Q^{p-1}),\\
(-LA_2)' + 2 p \theta (Q^{p-1})' + \alpha \sigma\Lambda Q + a_2 \sigma Q' = - p c_Q \, \p_x (e^{x}Q^{p-1})
\end{gathered}
\end{equation}
satisfying
\begin{equation}
\lim_{y \to + \infty} A_1 = \lim_{y \to + \infty} A_2 = 0, \qquad \lim_{y \to - \infty} A_1 = \lim_{y \to - \infty} A_2 = 2 \theta
\end{equation}
and
\[\int A_1  Q' = \int A_1 Q = 0, \quad \int A_2 Q' = \int (A_2 + 2 \theta) Q = 0 .\]
\end{enumerate}
\end{proposition}

\subsection{Proof of Proposition \ref{prop:lead}}
We compute $\EE_V$. 

Using $\frac{\p \tilde{R}_j}{\p \mu_j} = \Lambda \tilde{R}_j$, $\frac{\p \tilde{R}_j}{\p y_j} = - \p_y \tilde{R}_j$ and $\p_y^2 \tilde{R}_k - (1 + \mu_k) \tilde{R}_k + |\tilde{R}_k |^{p-1}\tilde{R}_k  = 0$, we have
\begin{multline}\label{eq:Er}
\EE_{\tilde{R}} = \dot{\mu}_1 \Lambda \tilde{R}_1 - (\dot{z}_1 - \mu_1)\p_y \tilde{R}_1 + \dot{\mu}_2\sigma \Lambda \tilde{R}_2 - (\dot{z}_2- \mu_2)\sigma\p_y \tilde{R}_2 \\ + \p_y [|\tilde{R}_1 + \sigma \tilde{R}_2|^{p-1} (\tilde{R}_1 + \sigma \tilde{R}_2) - \tilde{R}_1^p - \sigma \tilde{R}_2^p].
\end{multline}
Next, let
\[I(r) = \p_y \left[\p_y^2 r - r + p (\tilde{R}_1^{p-1} + \tilde{R}_2^{p-1}) r\right], \]
\[K_1(r, \varphi) = \mu_1 \frac{\p r}{ \p z_1}\varphi + \mu_2 \frac{\p r}{ \p z_2}\varphi + \mu_1  r \frac{\p \varphi}{\p z_1} + \mu_2 r\frac{\p\varphi}{\p z_2}, \]
\[K_2(r,\varphi)=p(\tilde{R}_1^{p-1} + \tilde{R}_2^{p-1}) r \p_y \varphi-  r \p_y\varphi +  r \p_y^3\varphi + 3\p_y r \p_y^2\varphi + 3 \p_y^2r \p_y\varphi\]
(note that $I(r), K_1(r, \varphi), K_2(r,\varphi)$ are linear in $r$) and
\[J_1 (\tilde{r}) = \dot{\mu}_1 \frac{\p \tilde{r}}{ \p \mu_1} + (\dot{z}_1 - \mu_1)\frac{\p \tilde{r}}{ \p z_1} + \dot{\mu}_2  \frac{\p \tilde{r}}{ \p  \mu_2} + (\dot{z}_2- \mu_2)\frac{\p \tilde{r}}{ \p z_2},\]
\[J_2 (\tilde{r}) =\p_y [|\tilde{R}_1 + \sigma \tilde{R}_2 + \tilde{r}|^{p-1} (\tilde{R}_1 + \sigma \tilde{R}_2 + \tilde{r}) - |\tilde{R}_1 + \sigma \tilde{R}_2|^{p-1} (\tilde{R}_1 + \sigma \tilde{R}_2) - p (\tilde{R}_1^{p-1} + \tilde{R}_2^{p-1}) \tilde{r}].\]
Then by computation, we see that
\begin{equation}\label{EV:total}
\EE_V = \EE_{\tilde{R}} + I(r)\varphi + K_1(r,\varphi) + J_1(\tilde{r})+ J_2(\tilde{r})+K_2(r,\varphi).
\end{equation}
\begin{lemma}[Expansion of nonlinear interaction]\label{lem:expnl} The nonlinear interaction term 
\[G = \p_y [|\tilde{R}|^{p-1} \tilde{R} - \tilde{R}_1^p - \sigma \tilde{R}_2^p]\] 
can be decomposed asymptotically as
\begin{equation}\label{eq:nonlinear}
G = p c_Q \left(\sigma e^{-z} \p_y [e^{- (y - z_1)}R_1^{p-1}] + e^{-z} \p_y [e^{(y - z_2)}R_2^{p-1}] \right) + \sum_{k=1}^2 O_{H^1}(|\mu_k| ze^{-z}) + O_{H^1}(e^{-\frac{3}{2}z}).
\end{equation}
\end{lemma}
\begin{proof}[Proof of Lemma \ref{lem:expnl}] We observe that for $z = z_1 - z_2$, by \eqref{asyQ},
\begin{align*}
Q (y) Q(y-z) \lesssim e^{-z}\qquad \mbox{and}& \qquad \|Q (y) Q(y-z) \|_{L^2} \lesssim z e^{-z} \\
|Q' (y)| Q(y-z) \lesssim e^{-z}\qquad \mbox{and}& \qquad \| Q' (y) Q(y-z) \|_{L^2} \lesssim z e^{-z} 
\end{align*}
so that (recall $p-1 \geq 2$)
\begin{align*}
\p_y [|\tilde{R}_1 + \sigma \tilde{R}_2|^{p-1} (\tilde{R}_1 + \sigma \tilde{R}_2)- \tilde{R}_1^p - \sigma\tilde{R}_2^p] &= p\p_y [\sigma\tilde{R}_1^{p-1} \tilde{R}_2 + \tilde{R}_1 \tilde{R}_2^{p-1} ] + O(e^{-z}R_1R_2)\\
&= p\p_y [\sigma\tilde{R}_1^{p-1} \tilde{R}_2 + \tilde{R}_1 \tilde{R}_2^{p-1} ] + O_{H^1}(ze^{-2z}).
\end{align*}
Next, also from the asymptotic behavior of $Q$ and $Q'$, we deduce the Taylor formula
\begin{multline}\label{Rtil}
\tilde{R}_k (t,y) = Q(y - z_k) + \mu_k \Lambda Q(y - z_k) + \mu_k^2 \int_0^1 (1-s) \Lambda^2Q_{1 + s\mu_k}(y - z_k) ds\\
= Q(y - z_k) + \mu_k \Lambda Q(y - z_k) + O (\mu_k^2(1 + |y-z_k|^2)e^{-|y-z_k|})
\end{multline}
thus we find
\begin{align*}
\sigma\p_y [\tilde{R}_1^{p-1} \tilde{R}_2] =& \sigma \p_y  (R_1 ^{p-1} R_2 ) + O_{H^1}(|\mu_1| ze^{-z}) + O_{H^1}(|\mu_2| ze^{-z}).
\end{align*}
On the other hand, we claim that
\[ \p_y  (R_1 ^{p-1} R_2 )= c_Q e^{-z} \p_y [e^{- (y - z_1)}R_1^{p-1}] +O_{H^1}(e^{-\frac{3}{2} z}).\]
Indeed, consider 
\[R_1 ^{p-1}\p_y R_2  + c_Q e^{-z }e^{-(y-z_1)}R_1^{p-1}  = R_1 ^{p-1}\p_y R_2  + c_Q e^{-(y-z_2)}R_1^{p-1}. \]
For $y < z_2$, by the exponential decay of $R_1$, we have $e^{-(y-z_2)} R_1 =e^{-z}$ and $|R_1 (y)|^{p} \lesssim e^{-p z}$ so as $p-1 \geq 2$,
\[\left|R_1 ^{p-1}\p_y R_2  + c_Q e^{-(y-z_2)}R_1^{p-1}\right| \lesssim e^{-\frac{3}{2} z} |R_1|^{\frac{1}{2}}.\]
For $y > z_2$, as $\p_y R_2 = -c_Q e^{- (y-z_2)} + O(e^{-2 (y-z_2)}) $, we also have:
\[\left|R_1 ^{p-1}\p_y R_2  + c_Q e^{-(y-z_2)}R_1^{p-1}\right| \lesssim e^{-\frac{3}{2} z} |R_1|^{\frac{1}{2}}.\]
By the same way, we consider $R_1 ^{p-2}\p_y R_1\, R_2  - c_Q e^{-z }e^{-(y-z_1)}R_1^{p-2} \p_y R_1$ and have
\[\left|R_1 ^{p-2}\p_y R_1 \,R_2  - c_Q e^{-z }e^{-(y-z_1)}R_1^{p-2} \p_y R_1\right| \lesssim e^{-\frac{3}{2}z} |R_1|^{\frac{1}{2}}.\]
Therefore,
\begin{align*}
\sigma\p_y [\tilde{R}_1^{p-1} \tilde{R}_2]= \sigma c_Q e^{-z} \p_y [e^{- (y - z_1)}R_1^{p-1}] + O_{H^1}(|\mu_1| ze^{-z}) + O_{H^1}(|\mu_2| ze^{-z})+ O_{H^1}(e^{-\frac{3}{2} z}).
\end{align*}
Similarly
\begin{align*}
\p_y [\tilde{R}_2^{p-1} \tilde{R}_1] = c_Q e^{-z} \p_y [e^{(y - z_2)}R_2^{p-1}]  +  O_{H^1}(|\mu_1| ze^{-z}) + O_{H^1}(|\mu_2| ze^{-z}) + O_{H^1}(e^{-\frac{3}{2} z}).
\end{align*}
\end{proof}

Now we construct the refined term $r(y; \Gamma (t))$ to match the order $e^{-z(t)}$ of the nonlinear interaction.
\begin{lemma}[Definition and equation of $r (t,y)$] \label{lem:A} There exist $\alpha > 0$, $\theta $, $a_1, a_2$, $\hat{A}_1 \in \mathcal{Y}, \hat{A}_2 \in \mathcal{Y}$ such that the two functions
\[A_1  = \hat{A}_1+ \theta \left(1 +\frac{Q'}{Q} \right), \qquad A_2 = \hat{A}_2 - \sigma \theta \left(1 +\frac{Q'}{Q} \right)
\]
solves
\begin{equation}\label{eq:A1}
(-LA_1)' - \alpha \Lambda Q - a_1 Q' =  - \sigma p c_Q \, \p_x (e^{-x}Q^{p-1}),\qquad \int A_1  Q' = \int A_1 Q = 0, 
\end{equation}
\begin{equation}\label{eq:A2}
(-LA_2)' + 2 p \theta (Q^{p-1})' + \alpha \sigma\Lambda Q + a_2 \sigma Q' = - p c_Q \, \p_x (e^{x}Q^{p-1}), \qquad \int A_2 Q' = \int (A_2 + 2 \theta) Q = 0.
\end{equation}
Moreover, by setting
\begin{equation}\label{eq:r}
r (y; (\mu (t), z(t))) = e^{- z(t)} [A_1(y - z_1(t)) + A_2 (y - z_2(t))]
\end{equation}
then
\begin{multline}\label{for:I}
\EE_{\tilde{R}} + I(r) \varphi=\vec{m}_1 \cdot \vec{M}_1 +  \vec{m}_2 \cdot \vec{M}_2  +  O_{H^1}(|\mu_1| z e^{-z}) +  O_{H^1}(|\mu_2| z e^{-z}) + O_{H^1} (e^{-\frac{5}{4} z}).
\end{multline}
\end{lemma}
\begin{remark}\label{remark4} Note that $\lim_{y \to + \infty} r(y)= 0 $ while the limit at $- \infty$ of $r(s,y)$ may be non-zero, in others words, the function $r(t,y)$ may have a tail on the left of two solitons which corresponds to a dispersion of size $e^{-z(t)}$ (in the integrable case $p=3$, we have $\theta = 0$ (see \eqref{theta}) so $r(y)$ has no tail, which is compatible to the property of integrable model).
\end{remark}
\begin{proof}[Proof of Lemma \ref{lem:A}] First, assume $A_1$ solves
\begin{equation}\label{A1:eq}
(-LA_1)' - \alpha_1 \Lambda Q - a_1 Q' =  - \sigma p c_Q \, \p_x (e^{-x}Q^{p-1})
\end{equation} 
and $A_2$ solves
\begin{equation}\label{A2:eq}
(-LA_2)' + 6 \theta_2 (Q^{p-1})' + \alpha_2 \sigma \Lambda Q + a_2 \sigma Q' = - p c_Q \, \p_x (e^{x}Q^{p-1}).
\end{equation}
To show $\alpha_1 = \alpha_2 = \alpha$, we multiply both sides of \eqref{A1:eq}, \eqref{A2:eq} with $Q$, integrate and use $L Q' =0$ and parity properties to obtain
\begin{align*}
&- \alpha_1 \int Q \Lambda Q  = - \sigma p c_Q \int (e^{-x}Q^{p-1})' Q =  \sigma c_Q \int e^{- x} Q^p
\end{align*}
and so
\begin{align*}
\alpha_ 1 = - \frac{\sigma c_Q\int e^{-x} Q^p }{\int Q \Lambda Q}.
\end{align*} 
Similarly, we also have
\begin{align*}
 \alpha_2 \sigma \int Q \Lambda Q  =  -p c_Q\int (e^{x}Q^{p-1})' Q = - c_Q\int e^{ x} Q^p = - c_Q \int e^{-x} Q^p
\end{align*} 
and so we deduce the unique possible value for $\alpha_1$ and $\alpha_2$
\[\alpha_ 1 =\alpha_2 = \alpha = -\frac{\sigma c_Q\int e^{-x} Q^p }{\int Q \Lambda Q}.\] 
Remark from \eqref{int:value} that in sub-critical cases $\int Q\Lambda Q > 0$ and in super-critical cases $\int Q \Lambda Q < 0$ thus by choice of sign of $\sigma$: in sub-critical cases $\sigma=-1$ and in super-critical cases $\sigma =1$, we have $\alpha > 0$ in both cases as required. By the parity of $Q$ and integration by parts, we obtain
\begin{align*}
\int e^{-x} Q^p dx =& \int \frac{e^{-x} + e^x}{2} Q^p dx= \int_0^\infty (e^{-x}+ e^x) Q^p dx \\
=& \lim_{R\to +\infty} \int_0^R (e^{-x}+ e^x) Q^p dx= \lim_{R\to +\infty} \int_0^R (e^{-x}+ e^x) (Q - Q'')dx\\
= & \lim_{R\to +\infty} \left( \int_0^R (e^{-x}+ e^x) Q\,dx + \int_0^R (-e^{-x}+ e^x) Q'dx - \left[(e^{-x}+e^x)Q'\right]\bigg|_0^R  \right)\\
= & \lim_{R\to +\infty} \left([(-e^{-x}+ e^x) Q]\bigg|_0^R - \left[(e^{-x}+e^x)Q'\right]\bigg|_0^R  \right) = 2\, c_Q =2 \left(2p+2\right)^{\frac{1}{p-1}}
\end{align*} 
using the asymptotic behavior \eqref{asyQ}. Combine with \eqref{int:value} to get
\begin{equation}\label{eq:alpha}
\alpha = \frac{ 8(p-1)}{|5-p|} \left(2p+2\right)^{\frac{2}{p-1}} \|Q\|_{L^2}^{-2}.
\end{equation}

Second, let us look for a solution $A_1$ under the form $A_1 = \hat{A}_1+ \theta_1 \left(1 +\frac{Q'}{Q} \right)$ then by \eqref{Q:iden}, \eqref{Q:iden2}, we deduce the equation of $\hat{A}_1$
\begin{multline*}
(- L\hat{A}_1)' - \theta_1 \left( - \frac{3p(p-1)}{p+1}Q^{p-1} + \frac{3 (3p-1)(p-1)}{(p+1)^2}Q^{2p-2}\right) \\- p \theta_1 (Q^{p-1})' - \alpha \Lambda Q - a_1 Q' = - \sigma p c_Q\p_x (e^{-x}Q^{p-1}).
\end{multline*}
To find $\hat{A}_1 \in \mathcal{Y}$, which implies $L \hat{A}_1 \in \mathcal{Y}$, we need to impose
\begin{equation}
\theta_1 \int \left(- \frac{3p(p-1)}{p+1}Q^{p-1} + \frac{3 (3p-1)(p-1)}{(p+1)^2}Q^{2p-2}\right)+ \alpha \int \Lambda Q = 0
\end{equation}
so from \eqref{int:value}, we get
\[\theta_1 = \frac{(p+1) \alpha \int \Lambda Q}{(p-1)\int Q^{p-1}}.\]
Similarly, we consider the equation of $\hat{A}_2$ and obtain the same equation for $\theta_2$
\begin{equation}
\sigma\theta_2 \int \left(- \frac{3p(p-1)}{p+1}Q^{p-1} + \frac{3 (3p-1)(p-1)}{(p+1)^2}Q^{2p-2}\right)+ \sigma\alpha \int \Lambda Q = 0
\end{equation}
thus 
\begin{equation}\label{theta}
\theta_1 = \theta_2 = \theta = \frac{(p+1) \alpha \int \Lambda Q}{(p-1)\int Q^{p-1}} = \frac{(p+1)(3-p)\alpha \int Q}{(p-1)^2 \int Q^{p-1}}.
\end{equation}
Next, let $Z \in \mathcal{Y}, \int Z Q' = 0$ be such that 
\[Z' = \theta \left(- \frac{3p(p-1)}{p+1}Q^{p-1} + \frac{3 (3p-1)(p-1)}{(p+1)^2}Q^{2p-2}\right) + \alpha \Lambda Q - \sigma p c_Q \p_x (e^{-x}Q^{p-1}).\]
Then it suffices to solve $ -  L (\hat{A}_1 + a_1 \Lambda Q) = Z$. Indeed, from properties of the linearized operator $L$, there exists unique $A \in \mathcal{Y}, \int A Q' = 0$ such that $ -LA =Z$. Therefore, we set $\hat{A}_1 = A - a_1 \Lambda Q$ solves the equation. We uniquely fix $a_1$ so that $\int \hat{A}_1 Q = 0$ as $\int Q \Lambda Q \neq 0$. It is straightforward to check that $A_1 = \hat{A}_1+ \theta \left(1 +\frac{Q'}{Q} \right)$ satisfies desired conditions. We do similarly for $A_2$. Note from the definition of $A_1, A_2$ that $A_1', A_2' \in \mathcal{Y}$ as $\hat{A}_1, \hat{A}_2 \in \mathcal{Y}$ and
\[\left(1 + \frac{Q'}{Q}\right)' = \frac{(p-1)Q^{p-1}}{p+1}  \in \mathcal{Y}.\]

Third, set
\[r (y; (\mu (t), z(t))) = e^{- z(t)} [A_1(y - z_1(t)) + A_2 (y - z_2(t))]. \]
By \eqref{Rtil}, we have
\begin{align*}
I(r) =& \p_y [\p_y^2 r - r + p (R_1^{p-1} + R_2^{p-1}) r] + O_{H^1}({|\mu_1| e^{-z}}) + O_{H^1}({|\mu_1| e^{-z}}).
\end{align*}
Consider 
\begin{multline} \p_y [\p_y^2 r - r + p (R_1^{p-1} + R_2^{p-1}) r]
= e^{- z} \left(-L A_1 \right)' (x- z_1) + e^{- z} \p_y\left(pR_1^{p-1} A_2(x-z_2) \right) \\+ e^{-z} \left(-L A_2 + 2p\theta Q^{p-1} \right)' (x- z_2) + e^{-z} \p_y\left( pR_2^{p-1}(A_1(x-z_1) - 2 \theta)\right).
\end{multline}
Using the estimates
\[|A_2(x-z_2)| \lesssim (1 + |x - z_2|^q)e^{-(x-z_2)}, \quad \mbox{ for } x>z_2 \]
and
\[|A_1(x-z_1) - 2 \theta| \lesssim (1 + |x - z_1|^q)e^{-(x-z_1)}, \quad \mbox{ for } x<z_1 \]
with $A_1', A_2' \in \mathcal{Y}$, we have that
			\begin{multline}I(r) =  e^{- z} \left(-L A_1 \right)' (x- z_1) + e^{-z} \left(-L A_2 + 2p\theta Q^{p-1} \right)' (x- z_2)\\ + O_{H^1}({|\mu_1| e^{-z}}) + O_{H^1}({|\mu_1| e^{-z}})+ \OO( e^{-\frac{3}{2}z}).
\end{multline}
Then, by \eqref{remark:varphi},
\[\|\varphi (t,y)/ (1 + e^{\frac{1}{2}(y - z_1(t))}) \|_{L^2} \lesssim e^{\frac{1}{4}z(t)}\] 
and we obtain
\begin{multline}
I(r) \varphi =  e^{- z} \left(-L A_1 \right)' (x- z_1)\psi (e^{-\frac{1}{2}z}y +1) + e^{-z} \left(-L A_2 + 2p\theta Q^{p-1} \right)' (x- z_2)\psi (e^{-\frac{1}{2}z}y +1)\\ + O_{H^1}({|\mu_1| e^{-z}}) + O_{H^1}({|\mu_1| e^{-z}})+ O_{H^1}( e^{-\frac{5}{4}z}).
\end{multline}

Forth, we deduce from \eqref{eq:Er}, the expansion of nonlinear interaction \eqref{eq:nonlinear}, the equation of $A_1$ \eqref{eq:A1} and $A_2$ \eqref{eq:A2} that 
\begin{align*}
&\EE_{\tilde{R}} + I(r) \varphi 
\\&= e^{-z}[\alpha \Lambda Q + a_1 Q' - \sigma p c_Q \p_y (e^{-y} Q^{p-1})] (y-z_1) \psi (e^{-\frac{1}{2}z}y +1)+\dot{\mu}_1 \Lambda \tilde{R}_1 - (\dot{z}_1 - \mu_1)\p_y \tilde{R}_1 \\
&+ e^{-z}[-\alpha \sigma\Lambda Q - a_2 \sigma Q' - p c_Q \p_y (e^{y} Q^{p-1})] (y-z_2) \psi (e^{-\frac{1}{2}z}y +1) + \dot{\mu}_2\sigma \Lambda \tilde{R}_2 - (\dot{z}_2- \mu_2)\sigma\p_y \tilde{R}_2\\ 
&+ \sigma  p c_Q e^{-z}\p_y (e^{-(y-z_1)}Q^{p-1}(y-z_1)) + p c_Q e^{-z}\p_y (e^{(y-z_2)}Q^{p-1}(y-z_2))\\
 &+ O_{H^1}(|\mu_1| z e^{-z}) +  O_{H^1}(|\mu_2| z e^{-z}) + O_{H^1} (e^{-\frac{5}{4} z})\\
&= (\dot{\mu}_1 +\alpha e^{-z})\Lambda \tilde{R}_1 - (\dot{z}_1 - \mu_1 - a_1 e^{-z})\p_y \tilde{R}_1\\
&+ (\dot{\mu}_2 - \alpha e^{-z})\sigma\Lambda \tilde{R}_2 -(\dot{z}_2 - \mu_2 +a_2 e^{-z})\sigma\p_y \tilde{R}_2\\
 &+ O_{H^1}(|\mu_1| z e^{-z}) +  O_{H^1}(|\mu_2| z e^{-z}) + O_{H^1} (e^{-\frac{5}{4} z})\\
&= \vec{m}_1 \cdot \vec{M}_1 +  \vec{m}_2 \cdot \vec{M}_2 + O_{H^1}(|\mu_1| z e^{-z}) +  O_{H^1}(|\mu_2| z e^{-z}) + O_{H^1} (e^{-\frac{5}{4} z})
\end{align*}
here we use ~\eqref{Rtil} and $\|Q(y - z_i)(1 - \varphi (y)) \|_{H^1}\lesssim \exp (- \frac{1}{4}e^{- \frac{1}{2}z}) \ll e^{- \frac{5}{4}z}, i =1,2$. Therefore, we obtain the estimation \eqref{for:I} as required.
\end{proof}
Finally, we will control other terms in \eqref{EV:total}. By the definition of $r(t,y)$, we have 
\[|r| + |\p_y r| +|\p_y^2 r| \lesssim e^{-z}.\]
Let consider $\frac{\p r}{\p z_k} = (-1)^k r - e^{-z} A_k' (y - z_k) $, as $A_k' \in \mathcal{Y}$ so we can control $\mu_k \frac{\p r}{\p z_k} \varphi$ as $\OO(|\mu_1| e^{-z})\varphi$. Moreover, recall \eqref{remark:varphi2}, $\left|\frac{\p \varphi(t,y)}{\p z_k}\right| \lesssim \left|\tilde{\varphi} (t,y)\right|$, we deduce that
\begin{align*}
K_1(r, \varphi)& =  \OO(|\mu_1| e^{-z})\left(\varphi +\tilde{\varphi}\right) + \OO(|\mu_2| e^{-z})\left(\varphi +\tilde{\varphi}\right) = \sum_{j=1}^2 O_{H^1}(|\mu_j| e^{-\frac{3}{4} z})
\end{align*}
since $\|\varphi (y)/ (1 + e^{\frac{1}{2}(y - z_1(t))}) \|_{L^2} + \|\tilde{\varphi} (t,y)/ (1 + e^{\frac{1}{2}(y - z_1(t))}) \|_{L^2} \lesssim e^{\frac{1}{4}z(t)}$. For $J_1(\tilde{r})$, note that $\tilde{r}$ does not depend on $\mu_1$, $\mu_2$ and by the product rule, the same way as we control $K_1(r,\varphi)$, we have $(\dot{z}_k - \mu_k) \frac{\p \tilde{z}}{\p z_k} = \OO (\left|\vec{m}_k\right|e^{-z}) \varphi$ thus
\[J_1(\tilde{r}) = \sum_{j=1}^2 \OO\left(|\vec{m}_j| e^{-z}\right)\left(\varphi +\tilde{\varphi} \right) =  \sum_{j=1}^2 O_{H^1}\left(|\vec{m}_j| e^{-\frac{3}{4} z}\right). \]
The term $J_2(\tilde{r})$ is quadratic in $\tilde{r}$ so $J_2(\tilde{r}) = O_{H^1}(e^{-2z})$. Recall from \eqref{remark:varphi1} that
\[\|\p_y^k \varphi (t,y)\|_{L^2} \lesssim  e^{\frac{1}{4}z(t)-\frac{k}{2}z(t)},\quad \mbox{for } k \in \mathbb{N}\]
so all terms in $K_2(r,\varphi)$ can be controlled in $H^1$ as $O_{H^1}\left(e^{-\frac{5}{4} z}\right)$. Therefore,
\begin{multline}
\|E \|_{H^1} \lesssim |\mu_1|z e^{-\frac{3}{4} z} + |\mu_2| z e^{- \frac{3}{4}z}+ e^{-\frac{5}{4} z} + \sum_{j=1}^2 \left|\vec{m}_j \right| \left( e^{-z} + |\mu_1| e^{-\frac{3}{4}z} + |\mu_2| e^{-\frac{3}{4}z}\right).
\end{multline}
The estimate \eqref{closeH:V} is a direct consequence of the definition of $r(s,y)$ (see Lemma \ref{lem:A}) and the choice \eqref{remark:varphi} of $\varphi (y)$.
\qed
\subsection{Modulation of the approximate solution} We state a standard modulation result around $V$ based on the Implicit Function Theorem (see e.g. Lemma 3.1 in \cite{MaMeinvent}) and we omit its proof.
\begin{lemma} [Modulation around $V$] \label{lem:mod} For $p \neq5$, there exist $\bar{\omega}_0 > 0 , \bar{z}_0>0, C>0$ such that if $w(t)$ is a solution of \eqref{rescale:eq} on some interval $I$ satisfying for some $0< \omega_0 < \bar{\omega}_0$, $z_0 > \bar{z}_0$
\begin{equation}\label{hypo:mod}
\forall t \in I, \qquad \inf_{z_1 - z_2 > z_0} \|w(t) - V(\cdot; (0,0, z_1, z_2)) \|_{H^1} \leq \omega_0.
\end{equation}
Then there exists a unique $\mathcal{C}^1$ function $\Gamma (t)= (\mu_1 (t), \mu_2(t), z_1(t), z_2(t))$ such that $w(t,y)$ decomposes on $I$ as
\begin{equation}
w(t,y) = V(y; \Gamma (t)) + \epsilon (t, y)
\end{equation}
which satisfies the orthogonality conditions
\begin{equation}\label{ortho}
\int \epsilon(t) \tilde{R}_1 (t) = \int \epsilon(t)\p_y \tilde{R}_1 (t) = \int \epsilon(t) \tilde{R}_2 (t) = \int \epsilon(t)\p_y \tilde{R}_2 (t) =0
\end{equation}
and for all $t\in I$
\begin{equation}
z(t) = z_1(t) - z_2(t) > z_0 - C \omega_0, \qquad \|\epsilon(t)\|_{H^1} + |\mu_1(t)|+|\mu_2(t)| \leq C \omega_0.
\end{equation}
Moreover, the equation of the rest term $\epsilon (t,y)$ writes
\begin{equation}\label{eq:eps}
\p_t \epsilon + \p_y \left(\p^2_y \epsilon - \epsilon + |V + \epsilon|^{p-1}(V +\epsilon) - |V|^{p-1}V \right) +  \vec{m}_1 \cdot \vec{M}_1 +  \vec{m}_2 \cdot \vec{M}_2  + E =0
\end{equation}
where $\vec{m}_k, \vec{M}_k$ and $E$ defined in Proposition \ref{prop:lead}.
\end{lemma}
Note that the choice of the special orthogonality conditions \eqref{ortho} is related to the coercivity property \eqref{coer:eq}.
\section{Backward uniform estimates}\label{sec:uni}
Let $(\mu^{in}, z^{in})\in \RR \times (0, + \infty)$ to be chosen with $0<\mu^{in} \ll 1, z^{in} \gg 1$. Let $u(t,x)$ be solution of \eqref{gkdv} with initial data
\begin{align*}
u (t^{in}, x) =& \, V(x - t^{in}; (\mu^{in}, - \mu^{in}, z_{in}, -z^{in})) + \epsilon^{in} (x - t^{in})\\
=&\,  \left[ Q_{1 + \mu^{in}} (\cdot - z^{in}) +\sigma Q_{1 - \mu^{in}} (\cdot + z^{in})  \right] (x - t^{in}) \\&+ e^{- 2z^{in}}\left[\left(A_1(\cdot- z^{in})+ A_2(\cdot +z^{in})\right)\psi(e^{-z^{in}}\cdot+1)\right](x - t^{in})+ \epsilon^{in} (x - t^{in})
\end{align*}
where $\sigma= -1$, $\epsilon^{in} \equiv 0$ for sub-critical cases while $\sigma =1$, $\epsilon^{in}$ chosen in an appropriate way with $\|\epsilon^{in}\|_{H^1} \leq C (t^{in})^{-\frac{3}{2}}$ for super-critical cases (see Section \ref{sec:super}). By the renormalization \eqref{renor}, we consider $w(t,y) = u(t,y+t)$ solution of \eqref{rescale:eq} on some open interval containing $t^{in}$, observe that
\begin{align*}
&w(t^{in}) = V(y; (\mu^{in}, - \mu^{in}, z_{in}, -z^{in})), \quad \epsilon (t^{in}) = \epsilon^{in}(y),\\
&\mu_1(t^{in}) = - \mu_2(t^{in}) = \mu^{in} , \quad z_1(t^{in}) = - z_2(t^{in}) = z^{in}.
\end{align*}
Denote $\bar{z} = z_1 + z_2, \bar{\mu} = \mu_1 + \mu_2 $, we claim the following uniform estimates:
\begin{proposition} [Uniform backward estimates] \label{esti:prop} There exists $t_0 \gg 1$ such that for all $t^{in} > t_0$, there is a choice of parameters ($\mu^{in}, z^{in}$) with
\begin{equation}
\mu_1(t^{in}) = - \mu_2(t^{in}) = \mu^{in} =\sqrt{\alpha} e^{- z^{in}}, \qquad z_1(t^{in}) = - z_2(t^{in}) = z^{in} \gg 1
\end{equation}
such that the solution $u$ of \eqref{gkdv} exists and satisfies the hypothesis \eqref{hypo:mod} of Lemma \ref{lem:mod} on the rescaled frame $(t,y)$. Moreover, the decomposition given in Lemma \ref{lem:mod} of $u$ satisfies the following uniform estimates, for all $t\in [t_0,t^{in}]$
\begin{equation}\label{esti:uni}
\begin{aligned}
\left|z(t) - 2\log(\sqrt{\alpha} t) \right| \lesssim t^{- \frac{1}{16}},\quad \left|\mu(t)\right| \lesssim t^{-1}\\ 
|\bar{\mu}(t)|\lesssim t^{- \frac{9}{8}}, \quad |\bar{z}(t)|\lesssim t^{- \frac{1}{16}}, \quad\|\epsilon (t)\|_{H^1} \lesssim t^{- \frac{9}{8}}
\end{aligned}
\end{equation}
where 
\begin{equation}
\alpha = - \frac{\sigma c_Q\int e^{-x}Q^p}{\int Q \Lambda Q}=\frac{ 8(p-1)}{|5-p|} \left(2p+2\right)^{\frac{2}{p-1}} \|Q\|_{L^2}^{-2} >0.
\end{equation}
\end{proposition}
Notice in Proposition \ref{esti:prop} that all estimates are independent of $t^{in}$, thus the distance between $u(t)$ and the approximate solution $V(t)$ depends only on $t$ and not on the time $t^{in}$ where $u(t)$ was taken equal to $V(t)+\epsilon^{in}$.
\subsection{Proof of the uniform estimates in sub-critical cases}\label{sec:sub}
\subsubsection{Bootstrap bounds} The proof of Proposition \ref{esti:prop} follows from bootstrapping the following estimates
\begin{equation}\label{boot1}
|\mu_1 + \mu_2| \leq t^{-\frac{9}{8}}
\end{equation}
\begin{equation}\label{boot2}
|z_1 + z_2| \leq t^{-\frac{1}{16}}
\end{equation}
\begin{equation}\label{boot3}
\|\e \|_{H^1} \leq t^{-\frac{9}{8}}
\end{equation}
\begin{equation}\label{boot4}
\left|  \frac{e^{\frac{1}{2}z}}{ \sqrt{\alpha}} - t \right|\leq t^{\frac{15}{16}}
\end{equation}
\begin{equation}\label{boot5}
\frac{1}{2} t^{-1} \leq \mu \leq 2 t^{-1}.
\end{equation}
The bootstrap regime implies immediately that
\begin{equation}\label{esti:z}
z (t) = 2 \log t + \log \alpha + O(t^{- \frac{1}{16}})
\end{equation}
and
\begin{equation}\label{esti:mu}
|\mu_1| + |\mu_2| \lesssim t^{-1}, \qquad |z_1| + |z_2| \lesssim \log t.
\end{equation}
For $t_0$ to be chosen large enough (independent of $t^{in}$), and all $t^{in}>t_0$, we define in view of Lemma \ref{lem:mod}:
\[t^* = \inf \{ \tau \in [t_0, t^{in}];\, \eqref{boot1}-\eqref{boot5} \mbox{ hold on }[\tau, t^{in}] \} .\]
\subsubsection{Control of modulation equation}
\begin{lemma}[Pointwise control of the modulation equations and the error]\label{lem:point} The following estimates hold on $[t_0, t^{in}]$
\begin{equation}\label{eq1}
\left|\dot{\mu}_1(t) + \alpha e^{-z(t)} \right| + \left|\dot{\mu}_2(t) - \alpha e^{-z(t)} \right| \lesssim t^{- \frac{9}{4}},
\end{equation}
\begin{equation}\label{eq2}
\left|\mu_j(t) - \dot{z}_j(t) \right| \lesssim t^{- \frac{9}{8}}.
\end{equation}
and
\begin{equation}\label{eqE}
\| E (t,\cdot)\|_{H^1} \lesssim t^{- \frac{9}{4}}.
\end{equation}
\end{lemma}
\begin{proof} [Proof of Lemma \ref{lem:point}]
We claim the following estimates for the modulation equations
\begin{equation}\label{eq:mu}
\left|\dot{\mu}_1(t) + \alpha e^{-z(t)} \right| + \left|\dot{\mu}_2(t) - \alpha e^{-z(t)} \right| \lesssim \|\epsilon (t) \|^2_{L^2} + z(t) e^{-z(t)}\|\epsilon\|_{L^2} + \|E\|_{L^2}
\end{equation}
\begin{equation}\label{eq:z}
\left|\mu_j(t) - \dot{z}_j (t)\right| \lesssim \|\epsilon \|_{L^2} + \|E \|_{L^2} + e^{- z(t)}.
\end{equation}
From \eqref{e:R}, \eqref{boot3}, \eqref{esti:mu}, we have
\begin{equation}\label{esti:E}
\| E\|_{H^1} \lesssim t^{- \frac{5}{2}}\log (t) + t^{- \frac{5}{2}} + \sum_{j= 1}^2 |\vec{m}_j|t^{-2}\lesssim t^{-\frac{9}{4}} + \sum_{j= 1}^2 |\vec{m}_j|t^{-2}.
\end{equation}
Therefore $\left|\dot{\mu}_1(t) + \alpha e^{-z(t)} \right| + \left|\dot{\mu}_2(t) - \alpha e^{-z(t)} \right| \lesssim t^{- \frac{9}{4}}, \left|\mu_j - \dot{z}_j \right| \lesssim t^{- \frac{9}{8}}$ follow from the bootstrap bounds on $\| \epsilon\|_{H^1}$ and $e^{-z(t)}$.
In order to prove \eqref{eq:mu}, \eqref{eq:z}, recall the equation of $\epsilon$, we have
\[
\p_t \epsilon + \p_y \left(\p^2_y \epsilon - \epsilon + |V + \epsilon|^{p-1}(V + \epsilon) - |V|^{p-1}V \right) + \sum_{j=1}^2\vec{m}_j\cdot  \vec{M}_j + E =0
\]
with 
\begin{multline*}
E = \p_y [|\tilde{R}_1 + \sigma \tilde{R}_2|^{p-1}(\tilde{R}_1 + \sigma \tilde{R}_2) - \tilde{R}_1^p - \sigma \tilde{R}_2^p] + I(r)\varphi + K(r,\varphi) + H(\tilde{r}) \\- e^{- \frac{1}{2} z(t)} r \varphi' + e^{- \frac{3}{2} z(t)} r \varphi''' + 3 e^{-  z(t)} \p_y r \varphi'' + 3 e^{- \frac{1}{2} z(t)}  \p_y^2r \varphi'.
\end{multline*}
From the orthogonality condition $\int \epsilon \tilde{R}_k =0$, we expand $\frac{\p}{\p t} \int \epsilon \tilde{R}_1$ and using the equation of $\epsilon (t)$ to obtain
\begin{align*}
0 =& \frac{\p}{\p t} \int \epsilon \tilde{R}_1\\
=& \int \epsilon \p_t \tilde{R}_1 + \int (\p_y^2 \epsilon - \epsilon + p |V|^{p-1} \epsilon) \p_y \tilde{R}_1 \\
&+ \int \left(|V + \epsilon|^{p-1}(V + \epsilon) - |V|^{p-1}V - p |V|^{p-1}  \epsilon \right)\p_y \tilde{R}_1\\
&- (\dot{\mu}_1 +\alpha e^{-z})\int\Lambda \tilde{R}_1 \tilde{R}_1 + \sigma(\mu_2 - \alpha e^{-z})\int\Lambda \tilde{R}_2 \tilde{R}_1 - \int E \tilde{R}_1.
\end{align*}
Using \eqref{closeH:V}, the equation of $Q_v$ and $\frac{\p \tilde{R}_j}{\p \mu_j} = \Lambda \tilde{R}_j$, $\frac{\p \tilde{R}_j}{\p y_j} = - \p_y \tilde{R}_j$, we get
\begin{align*}
0 =\frac{\p}{\p t} \int \epsilon \tilde{R}_1=& \dot{\mu}_1\int \epsilon \Lambda \tilde{R}_1 - (\dot{z}_1 - \mu_1)\int \epsilon \p_y \tilde{R}_1  + \|\epsilon \|_{L^2} O (z e^{-z}) + O(\|\epsilon\|_{L^2}^2)\\
&- (\dot{\mu}_1 +\alpha e^{-z})\int\Lambda \tilde{R}_1 \tilde{R}_1 + \sigma(\dot{\mu}_2 - \alpha e^{-z})O(z e^{-z}) - \int E \tilde{R}_1
\end{align*}
so for $z \gg 1$ and $\|\epsilon\|_{H^1} \ll 1$
\begin{multline}\label{e:1}
|\dot{\mu}_1 +\alpha e^{-z}| \leq C \left(\|\epsilon\|_{L^2}^2 + z e^{-z}\|\epsilon\|_{L^2}^2 + \|E\|_{L^2} \right)\\ + |\dot{z}_1 - \mu_1|O (\|\epsilon\|_{L^2}) + |\dot{\mu}_2 - \alpha e^{-z}| O(ze^{-z}).
\end{multline}
Next, we consider $\int \epsilon \p_y \tilde{R}_1 = 0$ so
\begin{align*}
0 =& \frac{\p}{\p t} \int \epsilon \p_y\tilde{R}_1\\
=& \int \p_y \epsilon \p_t  \tilde{R}_1 + \int (\p_y^2 \epsilon - \epsilon + p |V|^{p-1}  \epsilon) \p_y^2 \tilde{R}_1 \\
&+ \int \left(|V + \epsilon|^{p-1}(V + \epsilon) - |V|^{p-1}V - p |V|^{p-1}  \epsilon\right)\p_y^2 \tilde{R}_1\\
&+ (\dot{z}_1 - \mu_1 - a_1 e^{-z})\int\p_y \tilde{R}_1 \p_y\tilde{R}_1 +\sigma (\dot{z}_2 - \mu_2 + a_2 e^{-z})\int\p_y \tilde{R}_2 \p_y \tilde{R}_1 - \int E \p_y \tilde{R}_1\\
=& \dot{\mu}_1\int \p_y \epsilon \Lambda \tilde{R}_1 - (\dot{z}_1 - \mu_1)\int \p_y \epsilon \p_y \tilde{R}_1  + \|\epsilon \|_{H^1} O (|\mu_1|) + O(\|\epsilon\|_{L^2})\\
&+ (\dot{z}_1 - \mu_1 - a_1 e^{-z})\int\p_y \tilde{R}_1 \p_y\tilde{R}_1 + \sigma(\dot{z}_2 - \mu_2 + a_2 e^{-z})\int\p_y \tilde{R}_2 \p_y \tilde{R}_1 - \int E \p_y\tilde{R}_1
\end{align*}

so we deduce that
\begin{multline}\label{e:2}
|\dot{z}_1 - \mu_1| \leq |a_1| e^{-z} + C\| \epsilon\|_{H^1} \left(1+ |\mu_1|+ e^{-z} \right) + C \|E \|_{L^2}\\
+ |\dot{\mu}_1 +\alpha e^{-z}| O (\|\epsilon\|_{L^2})+ |\dot{z}_2 - \mu_2| O(z e^{-z}) + O (ze^{-2z}).
\end{multline}
Combining two estimates \eqref{e:1}, \eqref{e:2} with their analogues for $|\dot{\mu}_2 +\alpha e^{-z}|$ and $|\dot{z}_2 - \mu_2|$, the estimates \eqref{eq:mu},\eqref{eq:z} are proved. Finally, the estimate \eqref{eqE} is a direct consequence of \eqref{eq1}, \eqref{eq2} and \eqref{esti:E}.
\end{proof}
\subsubsection{Energy functional} We introduce a nonlinear energy functional for $\epsilon(t)$: choose $\rho = \frac{1}{32}$ and set
\[\phi (y) = \frac{2}{\pi} \arctan (\exp (8\rho y)) \]
so that $\displaystyle\lim_{y \to - \infty} \phi = 0$ and $\displaystyle\lim_{y \to + \infty} \phi = 1$. We see that $\forall y \in \RR$,
\[\phi (-y) = 1 - \phi (y), \qquad \phi'(y) = \frac{8\rho}{\pi \cosh (8 \rho y)},\]
\[|\phi''(y) |\leq 8 \rho |\phi'(y)|, \qquad |\phi '''(y)| \leq (8 \rho)^2 |\phi'(y)|.\]
Let 
$$\Phi_1 (t,y) = \frac{\phi(y)}{(1 + \mu_1(t))^2} + \frac{1 -\phi(y)}{(1+ \mu_2(t))^2}, \qquad\Phi_2 (t,y) = \frac{\mu_1(t)\phi(y)}{(1 + \mu_1(t))^2} + \frac{\mu_2(t)(1 -\phi(y))}{(1+ \mu_2(t))^2}$$ 
and consider
\begin{multline}
\WW (t) = \int \bigg[ \left((\p_y\epsilon)^2 + \epsilon^2 - \frac{2}{p+1} \left(|\epsilon+V|^{p+1} - |V|^{p+1} - (p+1)|V|^{p-1} V \epsilon\right) \right)\Phi_1(t,y)\\ + \epsilon^2 \Phi_2 (t,y)\bigg]dy.
\end{multline}
The functional $\WW$ is coercive in $\epsilon$ at the main order and it is an almost conserved quantity for the problem (see \cite{MaMeinvent}, \cite{RaSz11} for a similar functional).
\begin{proposition}[Coercivity and time control of energy functional]\label{prop:energy} For all $t \in [t^*, t^{in}]$,
\begin{equation}\label{coer:W}
\|\epsilon (t)\|_{H^1}^2 \lesssim \WW(t)
\end{equation}
and
\begin{equation}\label{conser:W}
\frac{\p}{\p t}\WW(t) \geq - C_0\, \left( t^{- \frac{9}{8}}\|\epsilon\|_{H^1}^2 + t^{-\frac{9}{4}}\|\epsilon\|_{H^1} \right)
\end{equation}
where $C_0>0$ a constant independent of $t^{in}$.
\end{proposition}
\begin{proof}[Proof of Proposition \ref{prop:energy}]
\item[(a)] The proof of the coercivity property \eqref{coer:W} is a standard consequence of \eqref{coer:eq} and the orthogonality properties \eqref{ortho} by an elementary localization argument. We refer to the proof of Lemma 4 in \cite{MMT2}. We observe that locally around each soliton $\tilde{R}_j$, the functional behaves essentially as
\[\int (\p_y \epsilon)^2 + (1 + \mu_j) \epsilon^2 - p\tilde{R}_j^{p-1}\epsilon^2,\]
which is a rescaled version of $\la L\epsilon, \epsilon \ra$.
\item[(b)] Now we compute $\frac{\p}{\p t}\WW (t)$
\begin{align*}
&\frac{1}{2} \frac{\p}{\p t} \WW (t) = \int \p_t \epsilon \left(- \p_y^2 \epsilon + \epsilon - \left(  |V + \epsilon|^{p-1} (V + \epsilon) - |V|^{p-1}V\right) \right)\Phi_1 \\
&- \int \p_t \epsilon\p_y\epsilon\p_y\Phi_1 + \int \p_t \epsilon  \,\epsilon \Phi_2 - \int \p_t V \left(  |V + \epsilon|^{p-1}(V + \epsilon) - |V|^{p-1}V - p |V|^{p-1}  \epsilon\right)\Phi_1\\
&+\frac{1}{2}\int \left((\p_y\epsilon)^2 + \epsilon^2 - \frac{2}{p+1} \left(|\epsilon+V|^{p+1} - |V|^{p+1} - (p+1)|V|^{p-1} V \epsilon\right) \right)\p_t\Phi_1 +\frac{1}{2}\int \epsilon^2 \p_t \Phi_2 .
\end{align*}

First, we consider
\[ W_1 (t) =  \int \p_t \epsilon \left(- \p_y^2 \epsilon + \epsilon - \left(  |V + \epsilon|^{p-1}(V + \epsilon) - |V|^{p-1}V\right) \right)\Phi_1.\]
Using the equation \eqref{eq:eps} of $\epsilon$ 
\[\p_t \epsilon = - \p_y \left(\p^2_y \epsilon - \epsilon + (|V + \epsilon|^{p-1}(V + \epsilon) - |V|^{p-1}V \right) -\sum_{j=1}^2  \vec{m}_j \cdot \vec{M}_j  -E,\]
we get
\begin{align*}
&W_1 =- \frac{1}{2}\int \left(- \p_y^2 \epsilon + \epsilon - \left(  |V + \epsilon|^{p-1}(V + \epsilon) - |V|^{p-1}V\right) \right)^2\p_y\Phi_1\\
&+\int E  \left( \p_y^2 \epsilon - \epsilon +   |V + \epsilon|^{p-1}(V + \epsilon) - |V|^{p-1}V\right)\Phi_1 + \sum_{j=1}^2 \int \vec{m}_j \cdot \vec{M}_j (\p_y^2 \epsilon - \epsilon + p \tilde{R}_j^{p-1} \epsilon)\Phi_1 \\
 &+ \sum_{j=1}^2 \int \vec{m}_j \cdot \vec{M}_j  \left(|V + \epsilon|^{p-1}(V + \epsilon) - |V|^{p-1}V - p |V|^{p-1}  \epsilon \right)\Phi_1+ \sum_{j=1}^2 \int \vec{m}_j \cdot \vec{M}_j (p|V|^{p-1}\epsilon - p\tilde{R}_j^{p-1}\epsilon)\Phi_1 .
\end{align*}
On the one hand, we have, by \eqref{eqE}
\[ \left|\int E  \left( \p_y^2 \epsilon - \epsilon +   |V + \epsilon|^{p-1}(V + \epsilon) - |V|^{p-1}V\right)\Phi_1 \right| \lesssim \|\epsilon \|_{H^1} \|E\|_{H^1} \lesssim t^{- \frac{9}{4}}\|\epsilon \|_{H^1}\]
and by \eqref{eq1}, \eqref{eq2}
\[\left|\int \vec{m}_j \cdot \vec{M}_j  \left(|V + \epsilon|^{p-1}(V + \epsilon) - |V|^{p-1}V - p |V|^{p-1}  \epsilon \right)\Phi_1 \right| \lesssim t^{-\frac{9}{8}}\|\epsilon\|_{H^1}^2 . \]
Using \eqref{esti:z} and the asymptotic bahavior of $Q$ \eqref{asyQ}, we obtain $\vec{M_j} (|V|^{p-1} - \tilde{R}_j^{p-1}) \lesssim |z|^q e^{-z} \lesssim t^{-\frac{3}{2}}$ so by pointwise control modulation equations \eqref{eq:mu}, \eqref{eq:z}
\[\left|\int \vec{m}_j \cdot \vec{M}_j (p|V|^{p-1}\epsilon - p\tilde{R}_j^{p-1}\epsilon)\Phi_1 \right| \lesssim t^{-\frac{3}{2}}\|\epsilon\|^2 .\]
On the other hand, by \eqref{eq1},
\[\left|\int (\dot{\mu}_1 +\alpha e^{-z}) \Lambda \tilde{R}_1 (\p_y^2 \epsilon - \epsilon + p \tilde{R}_j^{p-1} \epsilon)\Phi_1\right|\lesssim t^{-\frac{9}{4}} \|\epsilon\|_{H^1}.\]
Denote $L_j = - f'' + f - p \tilde{R}_j^{p-1} f$ then
\begin{align*}
L_j(\p_y \tilde{R}_j) =&  \mu_j \p_y \tilde{R}_j.
\end{align*}
Thus by \eqref{ortho}, \eqref{eq2} and remark that by the decay property of $Q$, $\phi', \phi''$
\begin{equation}\label{es:phi}
\begin{gathered}
\|e^{-\frac{1}{2}|y-z_j|} \p_y \Phi_1\|_{L^\infty} \lesssim (|\mu_1 |+ |\mu_2|) e^{- 2 \rho z} \lesssim t^{-\frac{9}{8}}\\
\|(\mu_j \Phi_1 -  \Phi_2)e^{-\frac{1}{2}|y-z_j|}\|_{L^\infty}\lesssim (|\mu_1 |+ |\mu_2|) e^{- 2 \rho z} \lesssim t^{-\frac{9}{8}}
\end{gathered}
\end{equation}
for $j=1,2$, here note that $\rho = \frac{1}{32}$, so we obtain
\begin{align*}
&\int (\dot{z}_1 - \mu_1 - a_1 e^{-z}) \p_y \tilde{R}_1 (\p_y^2 \epsilon - \epsilon + p \tilde{R}_j^{p-1} \epsilon) \Phi_1\\=& \int (\dot{z}_1 - \mu_1 - a_1 e^{-z}) L_1(\p_y \tilde{R}_1) \epsilon \Phi_1 + O (\|\epsilon\|_{H^1}^2 (|\mu_1 |+ |\mu_2|) e^{- 2 \rho z})\\
=&  (\dot{z}_1 - \mu_1 - a_1 e^{-z})\int \mu_1 \p_y \tilde{R}_1 \epsilon \Phi_1+ O (t^{-\frac{9}{8}}\|\epsilon\|_{H^1}^2)\\
= & (\dot{z}_1 - \mu_1 - a_1 e^{-z})\int  \p_y\tilde{R}_1 \epsilon \Phi_2+ O (t^{-\frac{9}{8}}\|\epsilon\|_{H^1}^2).
\end{align*} 
The same estimates hold for $\tilde{R}_2$ hence the first term $W_1$ of $\frac{\p}{\p t}\WW$ verifies
\begin{equation*}
\begin{aligned}
W_1(t) =- \frac{1}{2}\int \left(- \p_y^2 \epsilon + \epsilon - \left(  |V + \epsilon|^{p-1}(V + \epsilon) - |V|^{p-1}V\right) \right)^2\p_y\Phi_1\\
 + \sum_{j=1}^2(\dot{z}_j - \mu_j - a_j e^{-z})\int  \p_y\tilde{R}_j \epsilon \Phi_2 + O(t^{- \frac{9}{8}}\|\epsilon\|_{H^1}^2) + O(t^{-\frac{9}{4}}\|\epsilon\|_{H^1}).
\end{aligned}
\end{equation*} 
For the first term, using integration by parts, $\|e^{-\frac{1}{2}|y-z_j|}\p_y \Phi_1\|_{L^\infty} \lesssim (|\mu_1 |+ |\mu_2|) e^{- 2 \rho z} \lesssim t^{-\frac{9}{8}}$, $|\phi'''| \leq (8 \rho)^2 |\phi'| \leq \frac{1}{16}|\phi'|$ and the fact that since $ \mu_1(t) \geq \mu_2(t)$ so $\frac{1}{(1+\mu_1(t))^2} \leq \frac{1}{(1+\mu_2(t))^2}$, $\p_y \Phi_1 \leq 0$, we get
\begin{equation}\label{3}
\begin{aligned}
W_1(t) =&- \frac{1}{2}\int (- \p_y^2 \epsilon)^2 \p_y\Phi_1 - \frac{1}{2}\int \epsilon^2 \p_y \Phi_1 + \int \p_y^2 \epsilon \,\epsilon \p_y \Phi_1 + O(t^{- \frac{9}{8}}\|\epsilon\|_{H^1}^2)\\
 &+ \sum_{j=1}^2(\dot{z}_j - \mu_j - a_j e^{-z})\int  \p_y\tilde{R}_j \epsilon \Phi_2 + O(t^{- \frac{9}{8}}\|\epsilon\|_{H^1}^2) + O(t^{-\frac{9}{4}}\|\epsilon\|_{H^1})\\
 \geq & -\frac{3}{4} \int (\p_y\epsilon)^2 \p_y\Phi_1 - \frac{3}{8} \int  \epsilon^2 \p_y \Phi_1 + \sum_{j=1}^2(\dot{z}_j - \mu_j - a_j e^{-z})\int  \p_y\tilde{R}_j \epsilon \Phi_2\\
  &+ O(t^{- \frac{9}{8}}\|\epsilon\|_{H^1}^2) + O(t^{-\frac{9}{4}}\|\epsilon\|_{H^1}).
\end{aligned}
\end{equation} 

Second, consider
\[W_2= - \int \p_t \epsilon\, \p_y\epsilon\p_y\Phi_1 + \int \p_t \epsilon  \,\epsilon \Phi_2.\]
From the equation of $\p_t \epsilon$ \eqref{eq:eps}
\begin{equation}\label{4}
\begin{aligned}
\int \p_t \epsilon  \,\epsilon \Phi_2 =& \int \left( \p_y^2 \epsilon - \epsilon +   |V + \epsilon|^{p-1}(V + \epsilon) - |V|^{p-1}V\right) \p_y (\epsilon \Phi_2) - \int E \,\epsilon \Phi_2\\
&- \sum_{j=1}^2 \int \vec{m}_j \cdot \vec{M}_j \,\epsilon \Phi_2.
\end{aligned}
\end{equation}
We have
\[\left| \int E (\epsilon \Phi_2) \right| \lesssim \|\epsilon \|_{H^1} \|E\|_{H^1} \lesssim t^{- \frac{9}{4}}\|\epsilon \|_{H^1}\]
and from \eqref{eq1} $|\dot{\mu}_1 +\alpha e^{-z}| \lesssim t^{- \frac{9}{4}}$ so
\[\left|\int (\dot{\mu}_1 +\alpha e^{-z}) \Lambda \tilde{R}_1 (\epsilon \Phi_2) \right| \lesssim t^{-\frac{9}{4}} \|\epsilon \|_{H^1} |\mu_1| \lesssim t^{-3} \|\epsilon \|_{H^1}.\]
And for the first term, using integration by parts and the fact $\p_y \Phi_2 \geq 0$, we get
\begin{align*}
&\int \left( \p_y^2 \epsilon - \epsilon +   |V + \epsilon|^{p-1}(V + \epsilon) - |V|^{p-1}V\right) \p_y (\epsilon \Phi_2)\\ = &- \frac{3}{2} \int (\p_y \epsilon)^2 \p_y \Phi_2 -\frac{1}{2} \int \epsilon^2 \p_y \Phi_2 + \frac{1}{2} \int \epsilon^2 \p_y^3 \Phi_2 + \int \left(|V + \epsilon|^{p-1}(V + \epsilon) - |V|^{p-1}V\right) \p_y (\epsilon \Phi_2)\\
\geq & - \frac{3}{2} \int (\p_y \epsilon)^2 \p_y \Phi_2 -\frac{3}{4} \int \epsilon^2 \p_y \Phi_2 + \int \left(|V + \epsilon|^{p-1}(V + \epsilon) - |V|^{p-1}V\right) \p_y (\epsilon \Phi_2).
\end{align*}
As $\left( |V + \epsilon|^{p-1}(V + \epsilon) - |V|^{p-1}V\right) - p |V|^{p} \epsilon = O(\epsilon^2)$, let consider
\begin{align*}
&\int \left(|V + \epsilon|^{p-1}(V + \epsilon) - |V|^{p-1}V\right) \p_y (\epsilon \Phi_2) = \int p |V|^{p-1} \epsilon \p_y(\epsilon \Phi_2) + \sum_{j=1}^2|\mu_j|\|\epsilon\|_{H^1}^3 \\
= & - \int \frac{p(p-1)}{2} |V|^{p-3}V (\p_y V \Phi_2)\epsilon^2 - \left(\frac{p}{2} -1\right) \int p |V|^{p-1} \epsilon^2 \p_y \Phi_2 + O(s^{-2} \|\epsilon\|_{H^1}^2).
\end{align*}
However, by the decay property of $V$ and $\Phi$, we have
\[\| V \p_y \Phi_2\|_{L^\infty} \lesssim(|\mu_1 |+ |\mu_2|) e^{- 2 \rho z},\]
\[\left\|\Phi_2 \p_y V - \sum_{j=1}^2 \mu_j \p_y \tilde{R}_j\right\|_{L^\infty} \lesssim (|\mu_1 |+ |\mu_2|) e^{- 2 \rho z} + e^{-z}, \]
so gathering these computations
\begin{equation}\label{5}
\begin{aligned}
\int \p_t \epsilon \,\epsilon \Phi_2 \geq& - \frac{p(p-1)}{2} \sum_{j=1}^2 \mu_j \int \p_y \tilde{R}_j|V|^{p-3}V \epsilon^2 -\sum_{j=1}^2(\dot{z}_j - \mu_j - a_j e^{-z})\int  \p_y\tilde{R}_j \epsilon \Phi_2\\
&- \frac{3}{2} \int (\p_y \epsilon)^2 \p_y \Phi_2 -\frac{3}{4} \int \epsilon^2 \p_y \Phi_2+ O (s^{-2}\|\epsilon\|_{H^1}^2) + O (\|\epsilon\|_{H^1}^2 (|\mu_1 |+ |\mu_2|) e^{- 2 \rho z})\\
= & - \frac{p(p-1)}{2} \sum_{j=1}^2 \mu_j \int \p_y \tilde{R}_j|V|^{p-3}V \epsilon^2 -\sum_{j=1}^2(\dot{z}_j - \mu_j - a_j e^{-z})\int  \p_y\tilde{R}_j \epsilon \Phi_2\\
&- \frac{3}{2} \int (\p_y \epsilon)^2 \p_y \Phi_2 -\frac{3}{4} \int \epsilon^2 \p_y \Phi_2 + O (t^{-\frac{9}{8}}\|\epsilon\|_{H^1}^2).
\end{aligned}
\end{equation}
Moreover, by \eqref{eq:eps}, \eqref{es:phi}, integrating by parts and arguing as in \eqref{4}, we have
\begin{align*}
- \int \p_t \epsilon\,\p_y\epsilon\p_y\Phi_1 =& - \int \left(\p^2_y \epsilon - \epsilon + (|V + \epsilon|^{p-1}(V + \epsilon) - |V|^{p-1}V \right)\p_y\left(\p_y\epsilon\p_y\Phi_1\right)\\
 &+ \sum_{j=1}^2 \int \vec{m}_j \cdot \vec{M}_j  \,\p_y\epsilon\p_y\Phi_1  + \int E \,\p_y\epsilon\p_y\Phi_1\\
 \geq & -\int (\p_y^2 \epsilon)^2 \p_y\Phi_1 - \frac{7}{8} (\p_y \epsilon)^2 \p_y \Phi_1 + O (t^{-\frac{9}{8}}\|\epsilon\|_{H^1}^2) \geq O (t^{-\frac{9}{8}}\|\epsilon\|_{H^1}^2)
\end{align*}
as $\p_y \Phi_1 \leq 0$ and $\|E\|_{H^1} \lesssim t^{-\frac{9}{4}}, |m_j| \lesssim t^{- \frac{9}{8}}$, $\|e^{-\frac{1}{2}|y-z_j|} \p_y \Phi_1\|_{L^\infty} \lesssim t^{- \frac{9}{8}}$.
Therefore, we deduce that
\begin{multline}\label{2}
W_2 \geq - \frac{p(p-1)}{2} \sum_{j=1}^2 \mu_j \int \p_y \tilde{R}_j|V|^{p-3}V \epsilon^2 -\sum_{j=1}^2(\dot{z}_j - \mu_j - a_j e^{-z})\int  \p_y\tilde{R}_j \epsilon \Phi_2\\
- \frac{3}{2} \int (\p_y \epsilon)^2 \p_y \Phi_2 -\frac{3}{4} \int \epsilon^2 \p_y \Phi_2 + O (t^{-\frac{9}{8}}\|\epsilon\|_{H^1}^2).
\end{multline}

Next, let
\[W_3 =  - \int \p_t V \left(  |V + \epsilon|^{p-1}(V + \epsilon) - |V|^{p-1}V - p |V|^{p-1} \epsilon\right)\Phi_1.\]
Remark that from the definition of $V$ and $\phi$
\begin{multline}
\left\|\p_t V \Phi_1 - \left(\dot{\mu}_1 \Lambda \tilde{R}_1 - \dot{z}_1\p_y \tilde{R}_1 \right) - \sigma\left(\dot{\mu}_2 \Lambda \tilde{R}_2 - \dot{z}_2\p_y \tilde{R}_2 \right)\right\|_{L^\infty} \\
\lesssim |z|^q e^{-z} + (|\mu_1|+|\mu_2|)e^{- 2 \rho z} \lesssim t^{-\frac{9}{8}}
\end{multline}
as $|\dot{z}_j| \sim |\mu_j| \lesssim t^{-1}, |\dot{\mu}_j| \sim e^{-z} \lesssim t^{-2}$. And from the expansion of $|V+\epsilon|^{p-1}(V+\epsilon)$
\[\left\|\left(  |V + \epsilon|^{p-1}(V + \epsilon) - |V|^{p-1}V - p |V|^{p-1} \epsilon\right) - \frac{p(p-1)}{2}\epsilon^2 V|V|^{p-3}\right\|_{L^\infty} \lesssim \|\epsilon\|^3_{L^\infty}.\]
Combining with \eqref{boot1}-\eqref{boot5}, \eqref{eq1}, \eqref{eq2}, we have
\begin{equation*}
\begin{aligned}
&\left|W_3  - \frac{p (p-1)}{2} \dot{z}_1 \int \p_y \tilde{R}_1 V|V|^{p-3} \epsilon^2 - \frac{p (p-1)}{2} \dot{z}_2 \int \sigma \p_y \tilde{R}_2 V|V|^{p-3} \epsilon^2 \right|\\
\lesssim & \,\sum_{j=1}^2 |\dot{z_j}| \|\epsilon \|_{H^1}^3 + \sum_{j=1}^2 |\dot{\mu}_j |\|\epsilon \|_{H^1}^2 + t^{-\frac{9}{8}} \|\epsilon\|_{H^1}^2
\lesssim  t^{-\frac{9}{8}} \|\epsilon \|_{H^1}^2,
\end{aligned}
\end{equation*}
in other words,
\begin{equation}\label{1}
W_3  = \frac{p (p-1)}{2} \dot{z}_1 \int \p_y \tilde{R}_1 V|V|^{p-3} \epsilon^2 + \frac{p (p-1)}{2} \dot{z}_2 \int \sigma \p_y \tilde{R}_2 V|V|^{p-3} \epsilon^2 + O (t^{-\frac{9}{8}}\|\epsilon\|^2_{H^1}).
\end{equation}

Finally,
\[ W_4 = \frac{1}{2}\int \left((\p_y\epsilon)^2 + \epsilon^2 - \frac{2}{p+1} \left(|\epsilon+V|^{p+1} - |V|^{p+1} - (p+1)|V|^{p-1} V \epsilon\right) \right)\p_t\Phi_1 +\frac{1}{2}\int \epsilon^2 \p_t \Phi_2 .\]
Since 
\[\p_t \Phi_1 = \frac{-2(1 + \mu_1(t))\dot{\mu}_1(t)\phi(y)}{(1 + \mu_1(t))^3} + \frac{-2(1 + \mu_2(t))\dot{\mu}_2(t)(1 -\phi(y))}{(1+ \mu_2(t))^3},\]
\[\p_t\Phi_2 (t,y) = \frac{\left[(1 + \mu_1(t))\dot{\mu}_1 (t) - 2 \mu_1(t)\dot{\mu}_1(t)\right]\phi(y)}{(1 + \mu_1(t))^3} + \frac{\left[(1 + \mu_2(t))\dot{\mu}_2 (t) - 2 \mu_2(t)\dot{\mu}_2(t)\right](1 -\phi(y))}{(1+ \mu_2(t))^3},\] 
we get that 
\begin{equation*}
\left|W_4 \right| \lesssim (|\dot{\mu_1}|+ |\dot{\mu}_2|)\|\epsilon \|_{H^1}^2 \lesssim t^{-2} \|\epsilon \|_{H^1}^2
\end{equation*}
as $\left|\dot{\mu}_1 + \alpha e^{-z} \right| + \left|\dot{\mu}_2 - \alpha e^{-z} \right| \lesssim t^{- \frac{9}{4}}$ and $|e^{-z}| \lesssim t^{-2}$ so
\begin{equation}\label{6}
W_4 = O (t^{-2}\|\epsilon\|_{H^1}^2).
\end{equation}

To conclude, recall that by \eqref{eq2}, $\left|\mu_j(t) - \dot{z}_j(t) \right| \lesssim t^{- \frac{9}{8}}$ and remark that by explicit computations
\[\left|\p_y \Phi_1 + 2 \p_y \Phi_2\right| = \left|\frac{\mu_1^2(t) \phi'(y)}{(1+\mu_1(t))^2} + \frac{-\mu_2^2(t) \phi'(y)}{(1+\mu_2(t))^2}\right| \lesssim |\mu_1(t)|^2 + |\mu_2(t)|^2 \lesssim t^{-2}, \]
we can deduce from \eqref{3}, \eqref{2}, \eqref{1}, \eqref{6} that
\[\frac{d}{ds}\WW(t) = W_1(t) + W_2(t) + W_3(t) +W_4(t)  \geq -\, C_0 \left( t^{- \frac{9}{8}}\|\epsilon\|_{H^1}^2 + t^{-\frac{9}{4}}\|\epsilon\|_{H^1} \right)\]
for some $C_0>0$ as required.
\end{proof}
\subsubsection{End of the proof of Proposition \ref{prop:energy}} We close the bootstrap estimates \eqref{boot1}-\eqref{boot5}.
\\
\textbf{step 1} Closing the estimates in $\epsilon$ \eqref{boot3}. By \eqref{conser:W} in Proposition \ref{prop:energy}, we have
\[\frac{\p}{\p t}\WW(t) \geq - \,C_0 \left( t^{- \frac{9}{8}}\|\epsilon\|_{H^1}^2 + t^{-\frac{9}{4}}\|\epsilon\|_{H^1} \right) \geq - \,C t^{- \left(\frac{9}{4} + \frac{9}{8}\right)}.\]
Thus, by integration on $[t,t^{in}]$ for any $t \in [t^*, t^{in}]$, using $\epsilon (t^{in}) = 0$ in sub-critical cases (in super-critical cases $\|\epsilon (t^{in})\|_{H^1} \leq C(t^{in})^{-\frac{3}{2}}$, see Section \ref{sec:super} for details), we obtain
\[\WW (t) \lesssim t^{- \left(\frac{9}{4}+ \frac{1}{8}\right)}\]
so by \eqref{coer:W}
\[\|\epsilon\|_{H^1}^2 \lesssim t^{- \left(\frac{9}{4}+ \frac{1}{8}\right)}.\]
Therefore, for $t_0$ large enough, for all $t\in [t^*, t^{in}],$
\[\|\epsilon\|_{H^1}^2 \leq \frac{1}{2} t^{- \frac{9}{4}}\]
which strictly improves the estimates on $\|\epsilon\|_{H^1}$ in \eqref{boot3}.
\\
\smallskip
\textbf{step 2} Closing the parameters $\bar{\mu}, \bar{z}$. From the equations \eqref{eq1}, \eqref{eq2} and the estimates \eqref{boot1}, \eqref{boot4} in the bootstrap regime, we obtain
\[|\dot{\mu}_1 + \dot{\mu}_2| \lesssim t^{- \frac{9}{4}}, \qquad |\dot{z}_1 + \dot{z}_2| \lesssim t^{- \frac{9}{8}}.\]
By integrating on $[t, t^{in}]$, 
\[|\mu_1 + \mu_2| \leq t^{-\frac{5}{4}}, \qquad |z_1 + z_2| \leq t^{-\frac{1}{8}}\]
as we choose initial data
\[\mu_1 (t^{in})= - \mu_2 (t^{in}), \qquad z_1 (t^{in}) = - z_2 (t^{in}).\]
This improves the estimates \eqref{boot1}, \eqref{boot2}.
\\
\smallskip
\textbf{step 3} Closing the parameters $\mu, z$. Recall that by Lemma \ref{lem:point}
\[\left|\dot{\mu}_1(t) + \alpha e^{-z(t)} \right| + \left|\dot{\mu}_2(t) - \alpha e^{-z(t)} \right| \lesssim t^{- \frac{9}{4}},\]
\[\left|\mu_j (t) - \dot{z}_j  (t)\right| \lesssim t^{- \frac{9}{8}}.\]
Thus we deduce for $\mu(t) = \mu_1(t) - \mu_2(t)$ and $z(t) = z_1(t) - z_2(t) $
\[\left|\dot{\mu} + 2 \alpha e^{-z}\right| \lesssim t^{-\frac{9}{4}},\]
\[|\mu - \dot{z}| \lesssim t^{- \frac{9}{8}}.\]
We get
\[\left|\dot{\mu} \mu + 2 \alpha \dot{z} e^{-z}\right| \lesssim t^{- \left(3+\frac{1}{8}\right)}.\]
since $|\mu| \lesssim t^{-1} , |\dot{ \mu}| \lesssim t^{-2}$. Therefore, by explicit choice of initial data
\[\mu_1(t^{in}) = - \mu_2(t^{in}) =\sqrt{\alpha} e^{- z^{in}},\quad z_1(t^{in}) = - z_2(t^{in}) = z^{in}\]
then $\mu(t^{in}) = 2 \sqrt{\alpha} e^{-\frac{z(t^{in})}{2}}$, we integrate on $[t, t^{in}]$: for any $t \in [t^* ,t^{in}]$
\[\left|\mu^2 - 4\alpha e^{-z} \right| \lesssim t^{- \left( 2 + \frac{1}{8} \right)}.\]
Combining with \eqref{boot4}, $|\mu - 2 s^{-1}| \lesssim t^{- \left( 1 + \frac{1}{16} \right)}$ which closes \eqref{boot5}. Now, we need to finish the bootstrap argument for $z$ \eqref{boot4}. Let consider
\[\left|\mu - 2\sqrt{\alpha} e^{-\frac{1}{2} z} \right| +|\mu - \dot{z}| \lesssim t^{- \frac{9}{8}}\]
then we get
\[\left|\dot{z} - 2\sqrt{\alpha} e^{- \frac{1}{2}z}\right| \lesssim t^{-\frac{9}{8}}.\]
Note that $\frac{d}{dt}\left(e^{\frac{1}{2}z}\right) =\frac{1}{2} \dot{z}e^{\frac{1}{2}z}$ thus
\begin{equation}\label{dd}
\left|\frac{d}{dt}\left(e^{\frac{1}{2}z}\right) -  \sqrt{\alpha} \right| \lesssim t^{-\frac{1}{8}}
\end{equation}
here we use $|e^{-z}|\lesssim t^{-2}$. 

Next, we need to adjust the initial choice of $z^{in}\gg 1$ through a topological argument (see \cite{CMM} for a similar argument).
We define $\zeta $ and $\xi$ the following two functions on $[T^*,T^{in}]$
\be\label{d:zx}
\zeta(t) = \frac{e^{\frac 1 2 z}}{ \sqrt{\alpha}}  ,\quad
\xi(t) = (\zeta(t)-t)^2 t^{-\frac{15}{8}}.
\ee
Then,~\eqref{dd} writes
\be\label{ddd}
| \dot \zeta(t) - 1|\lesssim t^{-\frac{1}{8}}.
\ee
According to~\eqref{boot4}, our objective is to prove that there exists a suitable choice of 
\[\zeta(t^{in})=\zeta^{in}\in [t^{in}-(t^{in})^{\frac{15}{16}},t^{in}+(t^{in})^{\frac{15}{16}}],\]
so that $t^*=t_0$.
Assume for the sake of contradiction that for all $\zeta^{\sharp}\in [-1,1]$, the choice
\[
\zeta^{in} = t^{in} +  (t^{in})^{-\frac{15}{16}}\zeta^{\sharp}
\]
 leads to $t^*=t^*(\zeta^{\sharp})\in (t_0,t^{in}]$.
Since all estimates in the bootstrap regime except the one on $z$ have been strictly improved on $[s^*,s^{in}]$, it follows from  $t^*(\zeta^{\sharp})\in (t_0,t^{in}]$ and continuity that
\[
|\zeta(t^*(\zeta^{\sharp})) - t^*| = (t^*)^{\frac{15}{16}}\quad \hbox{i.e.} \quad
\zeta(t^*(\zeta^{\sharp})) = t^* \pm (t^*)^{\frac{15}{16}}.
\]
We need a transversality condition to reach a contradiction.
We compute:
\begin{equation}\label{xi}
\dot \xi(t) = 2 (\zeta(t)-t)(\dot \zeta(t) - 1) t^{-\frac{15}{8}}
- \frac{15}{8}(\zeta(t)-t)^{2} t^{-\frac{23}{8}}.
\end{equation}
At $t=t^{*}$, this gives
\[
\left| \dot \xi(t^{*}) + \frac{15}{8} (t^*)^{-1} \right|\lesssim (t^*)^{-\frac{17}{16}}.
\]
Thus, for $t_0$ large enough,
\be\label{ng}
\dot \xi(t^{*}) < - (t^*)^{-1}.
\ee
A consequence of the transversality property~\eqref{ng} is the continuity of the function $\zeta^{\sharp}\in[-1,1]\mapsto t^*(\zeta^{\sharp})$. Indeed, let $\epsilon > 0$ then there exists $\delta >0 $ such that $\xi (t^*(\zeta^\sharp) - \epsilon) > 1 + \delta$ and $\xi (t^*(\zeta^\sharp) + \epsilon) < 1 - \delta$. Moreover, by definition of $t^*(\zeta^\sharp)$ (choosing $\delta$ small enough) for all $t \in [t^*(\zeta^\sharp) + \epsilon, t^{in}]$ we have $ \xi (t) < 1 - \delta$. But from the continuity of the flow, there exists $\iota >0$ such that for all $|\tilde{\zeta}^\sharp - \zeta^\sharp| < \iota$
$$\forall t\in [t^{*}(\zeta^\sharp) - \epsilon, t^{in}], \quad |\tilde{\xi} (t) - \xi (t)| \leq \delta / 2$$
so we obtain that $t^*(\zeta^\sharp) - \epsilon \leq t^* (\tilde{\zeta}^\sharp) \leq t^*(\xi^\sharp)+\epsilon$ and the continuity of $t^* (\zeta^\sharp)$ as expected. Thus we deduce the continuity of the function $\Phi$ defined by 
\[
\Phi \, : \, \zeta^{\sharp}\in [-1,1] \mapsto 
(\zeta(t^{*})-t^{*})(t^{*})^{\frac{15}{16}}\in \{-1,1\}.
\]
Moreover, for $\zeta^{\sharp}=-1$ and $\zeta^{\sharp}=1$, in these two cases $\xi(t^{in})=1$, from ~\eqref{xi} we have that $\dot{\xi}(t^{in}) < 0$ thus $t^{*} =t^{in}$. Therefore, $\Phi(-1)=-1$ and $\Phi(1)=1$, but this is a contradiction with the continuity.

\smallskip

In conclusion, there exists at least a choice of 
 \[\zeta(t^{in})=\zeta^{in}\in \left(t^{in}-(t^{in})^{\frac{15}{16}},t^{in}+(t^{in})^{\frac{15}{16}}\right)\]
such that $t^*=t_0$. This concludes our bootstrap argument for \eqref{boot4}.
\subsection{Proof of the uniform estimates in super-critical cases}\label{sec:super} In this section, we present some modifications to prove the result in super-critical cases. Some extra parameters are needed in order to control the instability created by $Z^\pm$. Denote
\[\tilde{Z}^\pm_k (t,y) =Z^\pm_{1 + \mu_k (t)} (y - z_k(t)).\] 
Thus, instead of considering the final data
$u(t^{in}) = V(x - t^{in}; (\mu^{in}, - \mu^{in}, z_{in}, -z^{in}))$ as in sub-critical cases, we look at solution $u(t)$ of \eqref{gkdv} with final data
\begin{align*}
u (t^{in}, x) =& \, V(x - t^{in}; (\mu^{in}, - \mu^{in}, z_{in}, -z^{in})) + \epsilon^{in} (x - t^{in})
\end{align*}
where
\begin{multline}
\epsilon^{in} (y) = b_1^+ \tilde{Z}^+_1(t^{in}, y) + b_1^- \tilde{Z}^-_1(t^{in}, y) + b_2^+ \tilde{Z}^+_2(t^{in}, y)+ b_2^- \tilde{Z}^-_2(t^{in}, y)\\
+ b_3 \tilde{R}_1(t^{in}, y) + b_4\tilde{R}_2(t^{in}, y) + b_5 \p_y \tilde{R}_1(t^{in}, y) +b_6 \p_y \tilde{R}_2(t^{in}, y)
\end{multline}
and $b= (b_1^+, b_1^-, b_2^+, b_2^-,b_3, b_4, b_5,b_6)$ belongs to some small neighborhood of $0$ in $\R^8$.

We consider the decomposition of $w(t,y) = u(t, y+t)$ by Lemma \ref{lem:mod}
\[w(t,y) = V(y; \Gamma (t)) + \epsilon (t, y)\]
that satisfies the orthogonality conditions \eqref{ortho}. Define
\begin{equation}
a^\pm_k (t) = \int \epsilon (t,y) \tilde{Z}^\pm_k (t,y) dy, \qquad a^\pm(t) = (a^\pm_1(t), a^\pm_2(t)) .
\end{equation}
The following lemma allows us to establish a one-to-one mapping between the choice of $b= (b_1^+, b_1^-, b_2^+, b_2^-,b_3, b_4, b_5,b_6)$ and the initial constraints $a^+ (t^{in}) =0, a^-(t^{in})=a^{in}$.
\begin{lemma} [Modulated data in direction $Y^\pm$]\label{lem:mode} There exists $C > 0$ such that for all $t^{in} \geq t_0$ and for all $a^{in} =(a^{in}_1, a^{in}_2) \in B_{\R^2}(0,(t^{in})^{-\frac{3}{2}})$, there is a unique $b$ so that $||b|| \leq C \|a^{in}\|$ ($C$ independent of $t^{in}$) and the initial data satisfies
\begin{equation}
\begin{gathered}
\left(\int \epsilon^{in}(y) \tilde{Z}^-_1(t^{in}, y)dy , \int \epsilon^{in}(y) \tilde{Z}^-_2(t^{in}, y)dy \right)= (a^{in}_1,a^{in}_2),\\
\int \epsilon^{in} \tilde{Z}^+_1(t^{in})  = \int \epsilon^{in} \tilde{Z}^+_2(t^{in})  = \int \epsilon^{in} \tilde{R}_1 (t^{in}) = \int \epsilon^{in}\p_y \tilde{R}_1 (t^{in}) \\= \int \epsilon^{in} \tilde{R}_2 (t^{in}) = \int \epsilon^{in}\p_y \tilde{R}_2 (t^{in}) =0.
\end{gathered}
\end{equation}
\end{lemma}
\begin{proof}[Proof of Lemma \ref{lem:mode}] Denote
\[c = \left(\int \epsilon^{in} \tilde{Z}^-_1 ,\int \epsilon^{in} \tilde{Z}^+_1,\int \epsilon^{in} \tilde{R}_1 , \int \epsilon^{in}\p_y \tilde{R}_1   ,\int \epsilon^{in}\tilde{Z}^-_2,\int \epsilon^{in} \tilde{Z}^+_2,\int \epsilon^{in} \tilde{R}_2 ,\int \epsilon^{in}\p_y \tilde{R}_2\right)\]
and consider the linear maps
\begin{alignat*}{2}
\Psi:\, \RR^8 &\to H^1(\RR) &\qquad  \Phi : \, H^1(\RR) &\to \RR^8\\
b &\mapsto \epsilon^{in}  &\qquad \epsilon^{in}&\to c\\
\end{alignat*}
and $\Omega = \Phi \circ \Psi : \RR^8 \to \RR^8$. We can check that for some functions $A(y) , B(y) \in \mathcal{Y}$
\begin{align*}
\Omega = \Phi \circ \Psi &= 
   \left(
    \begin{array}{c;{2pt/2pt}r}
    \mbox{\LARGE $N$} &\mbox{\LARGE $0$}  \\ \hdashline[2pt/2pt]
    \mbox{\LARGE $0$} & \mbox{\LARGE $N$}
    \end{array}
    \right)
   + O(|\mu^{in}|)+ O(\left|\la A(y+z^{in} \vec{e}_1) , B(y) \ra\right|) \\
   &= \left(
    \begin{array}{c;{2pt/2pt}r}
    \mbox{\LARGE $N$} &\mbox{\LARGE $0$}  \\ \hdashline[2pt/2pt]
    \mbox{\LARGE $0$} & \mbox{\LARGE $N$}
    \end{array}
    \right)+ O(|\mu^{in}|) +  O(e^{- z^{in}})
\end{align*}
where $N$ is the Gramian matrix of $Z^\pm, Q, \p_y Q$ which are linearly independent. Indeed, $Z^+, Z^-, Q$ are linearly independent and orthogonal since they are eigenfunctions of $L\p_y$ corresponding to different eigenvalues $e_0, -e_0,0$. On the other hand, $\p_y Q$ are orthogonal to $Z^+, Z^-, Q$ (see Lemma 4.9 in \cite{VC1} for more properties of $Z^\pm$) so they are linearly independent. Thus $\det N \neq 0$ and with $z^{in} \gg 1, 0 < \mu^{in} \ll 1$, we have that $\Omega
$ is invertible around $0$. Therefore, for any $a^{in} \in B_{\R^2}(0,(t^{in})^{-\frac{3}{2}})$, we can choose 
$$b = \Omega^{-1}(( a^{in}_1, 0,0, 0,a^{in}_2, 0,0,0)), \qquad||b|| \leq ||\Omega^{-1}||\, |a^{in}|$$
to conclude the lemma.
\end{proof}
We now define the maximal time interval $[T(a^{in}), t^{in}]$ on which the bootstrap bounds \eqref{boot1}--\eqref{boot5} hold and 
\begin{equation}
\|a^\pm(t)\| \leq t^{- \frac{3}{2}}
\end{equation}
for all $t \in [T(a^{in}), t^{in}]$. The uniform backward estimates of Proposition \ref{esti:prop} state that there is a choice of ($\mu^{in}, z^{in}, a^{in}$) with
\begin{equation}
\mu^{in} =\sqrt{\alpha} e^{- z^{in}}, \qquad  z^{in} \gg 1, \qquad a^{in} \in B_{\R^2}(0,(t^{in})^{-\frac{3}{2}})
\end{equation}
such that $T(a^{in}) =t_0$. Indeed, we proceed as for sub-critical cases in Section \ref{sec:sub} and improve all estimates in the bootstrap bounds except those of $a^\pm (t)$ and $z(t)$. Remark that 
\[\epsilon(t^{in}) \lesssim \|b\|\]
so if we choose $a^{in} \in B_{\R^2}(0,(t^{in})^{-\frac{3}{2}})$, by Lemma \ref{lem:mode}, we have $\|\epsilon(t^{in})\|_{H^1}^2 \lesssim t^{-3}$ which is still enough to conclude 
\[\WW (t) \lesssim t^{- \left(\frac{9}{4}+ \frac{1}{8}\right)}\]
from the fact that $\frac{\p}{\p t}\WW(t) \geq - \,C_0 \left( t^{- \frac{9}{8}}\|\epsilon\|_{H^1}^2 + t^{-\frac{9}{4}}\|\epsilon\|_{H^1} \right) \geq - \,C t^{- \left(\frac{9}{4} + \frac{9}{8}\right)}.$ It seems to us that the reasoning to close the bootstrap bound of $z(t)$ still works, in fact, it is, however we will control $a^\pm (t)$ through a suitable value of $a^{in}$ also by a topological argument so we have to choose $(z^{in}, a^{in})$ in the same time. Now we claim the following preliminary estimates on the parameters $a^\pm(t)$.
\begin{lemma}\label{lem:10} For all $t \in [T(a^{in}), t^{in}]$,
\begin{equation}\label{e:a}
\bigg| \frac{d a^\pm}{dt} (t) \mp e_0 a^\pm (t) \bigg| \lesssim ||\epsilon||_{H^1}^2 + t^{-\frac{9}{4}}
\end{equation}
\end{lemma}
\begin{proof} [Proof of Lemma ~\ref{lem:10}] Recall the equation of $\epsilon$:
\[
\p_t \epsilon + \p_y \left(\p^2_y \epsilon - \epsilon + |V + \epsilon|^{p-1}(V + \epsilon) - |V|^{p-1}V \right) + \sum_{j=1}^2\vec{m}_j\cdot  \vec{M}_j + E =0
\]
and note that $\int \p_y \tilde{R}_k \tilde{Z}^\pm_k =0$ which follows from
\[\int \p_y Q Z^\pm  = \pm \frac{1}{e_0}\int \p_y Q L(\p_y Z^\pm) = \pm \frac{1}{e_0}\int \p_y L(\p_y Q) Z^\pm =0.\]
Then, we have
\begin{align*}
\frac{d a^\pm_k}{dt} (t)=& \int \p_t\epsilon \tilde{Z}^\pm_k  +\int \epsilon \p_t \tilde{Z}^\pm_k \\
=&\int (\p_y^2\epsilon - (1+\mu_k)\epsilon + p |V|^{p-1}\epsilon)\p_y \tilde{Z}^\pm_k \\
&+ \int \left(|V + \epsilon|^{p-1}(V + \epsilon) - |V|^{p-1}V - p |V|^{p-1}  \epsilon \right)\p_y \tilde{Z}^\pm_k\\
&- (\dot{\mu}_1 +\alpha e^{-z})\int\Lambda \tilde{R}_1 \tilde{Z}^\pm_k + \sigma(\mu_2 - \alpha e^{-z})\int\Lambda \tilde{R}_2 \tilde{Z}^\pm_k - \int E \tilde{Z}^\pm_k\\
&+ \dot{\mu}_k\int \epsilon \Lambda \tilde{Z}^\pm_k - (\dot{z}_k - \mu_k)\int \epsilon \p_y \tilde{Z}^\pm_k.
\end{align*}
Using $Z \in \mathcal{Y}$, for $k\neq j$
\[\int |\tilde{R}_j| (|\tilde{Z}^\pm_k|+|\p_y \tilde{Z}^\pm_k|) \lesssim z^q e^{-z} \]
hence denote $L_k= -f'' +(1+\mu_k)f - p \tilde{R}_k^{p-1}f$, we have
\[\int (\p_y^2\epsilon - (1+\mu_k)\epsilon + p |V|^{p-1}\epsilon)\p_y \tilde{Z}^\pm_k = \int \epsilon L_k (\p_y \tilde{Z}^\pm_k) + O (\|\epsilon\|_{H^1}e^{-\frac{3}{4}z}).\]
Moreover from \eqref{eq1}, \eqref{eq2}, \eqref{eqE}, we get
\begin{equation*}
\begin{gathered}
\left|(\dot{\mu}_1 +\alpha e^{-z})\int\Lambda \tilde{R}_1 \tilde{Z}^\pm_k\right| + \left|\sigma(\mu_2 - \alpha e^{-z})\int\Lambda \tilde{R}_2 \tilde{Z}^\pm_k\right| \lesssim \|\epsilon\|^2_{H^1},\\
\left|\int E \tilde{Z}^\pm_k\right| \lesssim t^{-\frac{9}{4}},\\
\left|\dot{\mu}_k\int \epsilon \Lambda \tilde{Z}^\pm_k\right| + \left|(\dot{z}_k - \mu_k)\int \epsilon \p_y \tilde{Z}^\pm_k\right| \lesssim \|\epsilon\|^2_{H^1}.
\end{gathered}
\end{equation*}
Using the equation of $Z^\pm$ \eqref{eqZ}, we obtain
\[\int \epsilon L_k (\p_y \tilde{Z}^\pm_k) = \pm e_0 (1 + \mu_k)^{\frac{3}{2}}\int \epsilon \tilde{Z}^\pm_k.\]
As $|\mu_k| \lesssim t^{-1}, |a^\pm| \lesssim t^{-\frac{3}{2}}$, we get
\[\frac{d a^\pm_k}{dt} (t)= \pm e_0 a^\pm_k (t)+ O(||\epsilon||_{H^1}^2) + O(t^{-\frac{9}{4}})\]
as required.
\end{proof}
We now control $a^\pm(t)$ through topological arguments by noticing that the direction $a^+(t)$ is already stable. Consider $\zeta(t) = \frac{e^{\frac 1 2 z}}{ \sqrt{\alpha}}$ and $\xi(t)$ as defined in~\eqref{d:zx}.
\begin{lemma} [Control of $a^\pm(t)$]\label{lem:11} There exist $\zeta^{in} = \zeta (t^{in}) \in [t^{in} -  (t^{in})^{-\frac{15}{16}}, t^{in} +  (t^{in})^{-\frac{15}{16}}]$ (in consequence, $z^{in} \gg 1$) and $a^{in} \in B_{\R^2}(0,(t^{in})^{-\frac{3}{2}})$ such that $T(a^{in}) = t_0$.
\end{lemma}
\begin{proof}  [Proof of Lemma ~\ref{lem:11}] First of all, we claim that for all $a^{in} \in B_{\R^2}(0,(t^{in})^{-\frac{3}{2}}) $, the following inequality holds for all $ t\in [T(a^{in}), t^{in}]$
\begin{equation}
\left|a^+(t)\right| \leq \frac{1}{2}t^{-\frac{3}{2}}.
\end{equation}
Indeed, it follows the bootstrap bounds, ~\eqref{e:a} and $a^+(t^{in}) =0$ that for all $t \in [T(a^{in}), t^{in}]$
\begin{align*}
|a^+ (t)| &\lesssim  e^{e_0t}\int_t^{t^{in}} e^{- e_0\tau} \tau^{-\frac{9}{4}} d\tau \\
& = \frac{e^{e_0t}}{e_0}[e^{-e_0t}t^{-\frac{9}{4}} - e^{-e_0t^{in}}(t^{in})^{-\frac{9}{4}}] -  \frac{9e^{e_0t} }{4e_0} \int_t^{t^{in}} e^{- e_0\tau} \tau^{-\frac{13}{4}} d\tau\\
&\leq\frac{1}{e_0} t^{-\frac{9}{4}} \leq \frac{1}{2} t^{-\frac{3}{2}}
\end{align*}
for $t_0$ to be large enough. 

Let $\mathbb{D} = [-1,1]\times B_{\R^2}(0,1)$ equipped with the norm $\|(x,y)\|= \max(\|x\|, \|y\|)$. Now we suppose that for all $(\zeta^{\sharp}, a^{\sharp})\in \mathbb{D}$, the choice
\[
\zeta^{in} = t^{in} +  (t^{in})^{-\frac{15}{16}}\zeta^{\sharp},\quad
a^{in} = a^{\sharp}(t^{in})^{-\frac{3}{2}}
\]
gives us $T(a^{in}) =T(\zeta^{\sharp}, a^{\sharp}) \in (t_0, t^{in}]$. Recall that
\begin{equation}\label{xi1}
\dot \xi(t) = 2 (\zeta(t)-t)(\dot \zeta(t) - 1) t^{-\frac{15}{8}}
- \frac{15}{8}(\zeta(t)-t)^{2} t^{-\frac{23}{8}}.
\end{equation}
On the other hand, consider
$$\mathcal{N}(t) = t^{3}\|a^-(t)\|^2$$
then for $t \in (T(\zeta^{\sharp}, a^{\sharp}), t^{in}]$, by the bound on $||\epsilon||_{H^1}^2$ and ~\eqref{e:a}, we have
\begin{align*}
\dot{\mathcal{N}} (t)&= t^{3} \la 3 t^{-1}a^-(t) + 2 \frac{da^-}{dt}(t),a^-(t)\ra\\
&= t^{3} (3t^{-1} - 2 e_0) \|a^-(t)\|^2 + O\left(t^{-\frac{3}{4}}t^{\frac{3}{2}}\|a^-(t)\|\right).
\end{align*}
Therefore, with $t_0$ large enough ($\frac{3}{t_0} < \frac{1}{2}e_0$), we deduce that
\begin{equation}\label{e:N}
\dot{\mathcal{N}} (t) \leq -\frac{3}{2} e_0 \mathcal{N} (t)  + Ct^{-\frac{3}{4}}\sqrt{\mathcal{N}(t)}.
\end{equation}
Denote 
$$\Psi_1 (t) = (\zeta(t)-t)(t)^{\frac{15}{16}},$$
$$\Psi_2 (t) = a^-(t) t^{\frac{3}{2}}.$$
From the definition of $T(a^{in})$ and the continuity of flow, at the limit $T(\zeta^{\sharp}, a^{\sharp})$, we have one of the following situation
\begin{equation}
\Psi_1 (S(\zeta^{\sharp}, a^{\sharp})) = \pm 1, \qquad \Psi_2 \in B_{\R^2}(0,1)
\end{equation}
or
\begin{equation}
\|\Psi_2 (S(\zeta^{\sharp}, a^{\sharp})) \|= 1 \Leftrightarrow \Psi_2 \in \p B_{\R^2}(0,1), \qquad \Psi_1 \in [-1, 1]
\end{equation}
where $\p B_{\R^2}(0,1)$ is the boundary of $B_{\R^2}(0,1)$. Remark that in the first case, we have 
$$\dot{\xi} (T(\zeta^{\sharp}, a^{\sharp})) < - (T(\zeta^{\sharp}, a^{\sharp}))^{-1} < 0$$
and in the second case we have $\mathcal{N}(T(\zeta^{\sharp}, a^{\sharp}))=1$
$$\dot{\mathcal{N}} (T(\zeta^{\sharp}, a^{\sharp})) \leq - e_0 <0.$$
A consequence of the above transversality property is the continuity of $(\zeta^{\sharp}, a^{\sharp}) \mapsto T((\zeta^{\sharp}, a^{\sharp}))$ thus the following map
\begin{displaymath}\begin{array}{ccccc}
\Psi & : & \mathbb{D} & \to & \p \mathbb{D} \\
 & & (\zeta^{\sharp}, a^{\sharp}) & \mapsto & (\Psi_1 (T(\zeta^{\sharp}, a^{\sharp})), \Psi_2 (T(\zeta^{\sharp}, a^{\sharp})))\end{array}\end{displaymath}
is also continuous where $\p \mathbb{D}$ is the boundary of $\mathbb{D}$. Note that if $a^{\sharp} \in \p B_{\R^2}(0,1)$, then from ~\eqref{e:N}, $\dot{\mathcal{N}} (t^{in}) < 0$, we have $T(\zeta^{\sharp}, a^{\sharp}) = t^{in}$ and if $\zeta^{\sharp} =\pm 1$, then from ~\eqref{xi1}, $\dot{\xi}(t^{in}) <0$, we also have $T(\zeta^{\sharp}, a^{\sharp}) = t^{in}$. Thus $\Psi (\zeta^{\sharp}, a^{\sharp}) = (\zeta^{\sharp}, a^{\sharp})$ for all $(\zeta^{\sharp}, a^{\sharp}) \in \p \mathbb{D}$, which means that the restriction of $\Psi$ to the boundary of $\mathbb{D}$ is the identity. But the existence of such a map contradicts the Brouwer fixed point theorem. In conclusion, there exists a final data $(z^{in}, a^{in})$ such that $T(a^{in}) = t_0$, which concludes the proof of Proposition \ref{esti:prop} in super-critical cases.
\end{proof}
\section{Construction of solution}\label{s:3}
Applying Proposition \ref{esti:prop} with $t^{in}= n$ for any $n$ large enough, there exists a solution $u_n(t)$ of \eqref{gkdv} on the interval $[t_0, n]$ whose decomposition 
\[(\Gamma^n(t); \epsilon_n (t)) = ((z_1^n (t), z_2^n(t), \mu_1^n (t) , \mu_2^n (t)); \epsilon_n (t))\]
satisfies the uniform estimates \eqref{esti:uni}. Denote
\[\tilde{N}_n (t,x) =  Q_{1 + \mu_1^n(t)} (x -  t - z_1^n(t)) + \sigma Q_{1+\mu_2^n(t)}(x  - t - z_2^n(t)).\]
From \eqref{closeH:V} $\|V_n (t,x) - \tilde{N}_n (t,x)\|_{H^1} \lesssim t^{-\frac{3}{2}} \log^q (t)$ and \eqref{esti:uni} $\|\epsilon_n (t)\|_{H^1} \leq t^{- \frac{9}{8}}$, we have
\begin{equation}\label{remark1}
\|u_n(t) - \tilde{N}_n (t,x)\|_{H^1} \lesssim t^{- \frac{9}{8}}.
\end{equation}
On the other hand, by setting
\begin{equation}
N(t,x) = Q(x-t- \log(\sqrt{\alpha} t)) + \sigma Q(x-t+\log(\sqrt{\alpha} t))
\end{equation}
we deduce from \eqref{esti:uni} that
\begin{equation}\label{remark2}
\begin{aligned}
&\|\tilde{N}_n (t,x) - N(t,x)\|_{H^1} \lesssim |\mu_1^n (t)|+|\mu_1^n (t)|+|z_1^n (t) - \log (\sqrt{\alpha} t)|+|z_2^n (t) + \log (\sqrt{\alpha} t)|\\
\lesssim & \left| \frac{\bar{\mu}(t) + \mu(t)}{2} \right|+\left| \frac{\bar{\mu}(t) + \mu(t)}{2} \right|+ \left| \frac{\bar{z}(t) + z(t)}{2} - \log (\sqrt{\alpha} t)\right| + \left| \frac{\bar{z}(t) - z(t)}{2} - \log (\sqrt{\alpha} t)\right|\\
\lesssim & \,\,t^{- \frac{1}{16}}.
\end{aligned}
\end{equation}
Therefore, there exist a sequence of backward solutions $u_n \in \mathcal{C} ([t_0, n], H^1)$ of \eqref{gkdv} such that for all $t \in [t_0 , n],$
\begin{equation}\label{esti:main}
\|u_n(t)- N(t,x)\|_{H^1} \lesssim t^{- \frac{1}{16}}.
\end{equation}
\begin{remark}\label{remark5} We see that the size of the extra term $\tilde{r} (t,y)$ in the definition  \eqref{V0:eq} of the approximate solution $V(t,y)$ is much smaller than the estimate on $\epsilon(t,y)$. However, by Lemma \ref{lem:point}, the term
\[\tilde{r} (t,y)=e^{- z(t)}\left(A_1(y - z_1(t)) + A_2 (y - z_2(t))\right)\varphi (t,y)\]
improves the computation of the error to the flow \eqref{flow} to obtain \eqref{eqE}. This refinement is essential to close the bootstrap \eqref{boot3} on $\epsilon(t,y)$, see Proposition \ref{prop:energy}, since without it, one would obtain $E$ of size $t^{-2}$ and $\epsilon$ of size $t^{-1}$, which is not sharp enough to exploit the modulation equations \eqref{eq1}, \eqref{eq2}.
\end{remark}
Next, we  construct a function $u_0 \in H^1 (\RR)$ as a strong limit of a subsequence of $u_n(t_0)$.
\begin{lemma}\label{l:comp}
There exist $u_0 \in H^1 (\RR)$ and a sub-sequence, still denoted $u_n$, such that
$$u_n (t_0) \rightharpoonup u_0 \mbox{ weakly in } H^1(\RR)$$
$$u_n (t_0) \rightarrow u_0 \mbox{ in } H^\sigma(\RR), \mbox{ for } 0 \leq \sigma <1$$
as $n \to \infty$.
\end{lemma}
\begin{proof} [Proof of Lemma ~\ref{l:comp}]
By the bounds on $u_n$ and interpolation, it is enough to prove that the sub-sequence $u_n(t_0) \xrightarrow{L^2} u_0$ as $n\to \infty$. First, we claim the following: $\forall \delta_1 > 0, \delta_1 \ll 1$, $\exists n_0 \gg 1$, $ \exists K_1 = K_1(\delta_1) >0$ such that $\forall n \geq n_0$
\begin{equation}\label{claim}
\int_{|x| > K_1} |u_n(t_0, x)|^2 dx < \delta_1.
\end{equation}
Indeed, by \eqref{esti:main}, we have, for all $n$
\[\|u_n(t)- N(t,x)\|_{H^1} \lesssim t^{- \frac{1}{16}}.\]
A direct consequence of the above estimate is
\begin{equation}\label{b:un}
||u_n(t)||_{H^1} < C
\end{equation}
for all $t \in [t_0, n]$ since $||N(t)||_{H^1} \leq 2 ||Q||_{H^1} $. Furthermore, for fixed $\delta_1$, there exists $t_1 > t_0$ such that 
$$||u_n(t_1) - N(t_1)||_{H^1} \lesssim (t_1)^{-\frac{1}{16}} < \sqrt{\delta_1} $$
for $n$ large enough that $n > t_1$; in others words, we have
$$\int |u_n(t_1,x) - N(t_1,x)|^2 dx < \delta_1.$$
Besides, for $K_2 \gg 1$ large enough we have
$$\int_{|x| > K_2} |N(t_1,x)|^2 dx < \delta_1.$$
Consider now a $\mathcal{C}^1$ cut-off function $g : \RR \to [0,1]$ such that : $g \equiv 0$ on $(- \infty, 1]$, $0 < g' < 2$ on $(1,2)$ and $g \equiv 1$ on $[2, + \infty)$. Since $||u_n(t)||_{H^1} < C$ bounded in $H^1$ independently of $n$ and $t \in [t_0, n]$, we can choose $\gamma_1 > 0$ independent of $n$ such that
$$\gamma_1 \geq \frac{2}{\delta_1} (t_1-t_0) C^2.$$
We have by direct calculations, for $t \in [t_0, n]$
\begin{align*}
\left| \frac{d}{dt} \int |u_n(t, x)|^2 g\left( \frac{|x| -K_2}{\gamma_1} \right) dx \right|& = \left| \frac{1}{\gamma_1} \im \int u \left( \nabla \bar{u}\cdot \frac{x}{|x|} \right) g' \left( \frac{|x| - K_2}{\gamma_1} \right) \right|\\
&\leq \frac{2}{\gamma_1} \sup_{n \geq t\geq t_0}||u_n(t)||_{H^1}^2 \leq \frac{\delta_1}{t_1 - t_0}.
\end{align*}
By integration from $t_0$ to $t_1$
\begin{align*}
\int |u_n(t_0, x)|^2 g\left( \frac{|x| -K_2}{\gamma_1} \right) dx - \int |u_n(t_1, x)|^2 g\left( \frac{|x| -K_2}{\gamma_1} \right) dx \\\leq \int_{t_0}^{t_1} \left| \frac{d}{dt} \int |u_n(t, x)|^2 g\left( \frac{|x| -K_2}{\gamma_1} \right) dx \right| \leq \delta_1 .
\end{align*}
From the properties of $g$ we conclude:
\begin{align*}
&\int_{|x| > 2 \gamma_1 + K_2} |u_n(t_0, x)|^2 dx \leq \int |u_n(t_0, x)|^2 g\left( \frac{|x| -K_2}{\gamma_1} \right) dx \\\leq &\int |u_n(t_1, x)|^2 g\left( \frac{|x| -K_2}{\gamma_1} \right) dx + \delta_1 \leq \int_{|x| >  K_2} |u_n(t_1, x)|^2 dx + \delta_1 \leq 5 \delta_1 .
\end{align*}
Thus ~\eqref{claim} is proved. As $||u_n(t_0)||_{H^1} < C$, there exists a subsequence of $(u_n)$ (still denoted by $(u_n)$)
and $u_0\in H^1$ such that 
\[\begin{aligned}
&u_n(t_0) \rightharpoonup u_0 \quad \hbox{weakly in $H^1(\RR)$,}\\
&u_n(t_0) \rightarrow u_0 \quad \hbox{ in $L^2_{loc}(\RR)$, as $n\to +\infty$}
\end{aligned}\]
and by ~\eqref{claim}, we obtain that $u_n(t_0) \xrightarrow{H^\sigma} u_0$, for $0\leq \sigma <1$.
\end{proof}
To conclude the proof of Main Theorem, we consider $u(t)$ the solution of \eqref{gkdv} corresponding to $u (t_0) = u_0$. From \cite{KPV}, we have the continuous dependence of the solution upon the initial data, so for all $t \in [t_0, + \infty)$,
\[u_n(t) \to u(t) \quad \mbox{ in } H^\sigma (\RR),\quad s_c \leq \sigma <1\]
\[u_n(t) \rightharpoonup u(t)\quad \mbox{ in } H^1 (\RR).\]
where $s_c<1$ is the critical exponent. Thus, from \eqref{esti:main}, $\|u_n(t)- N(t,x)\|_{H^1} \lesssim t^{- \frac{1}{16}}$, passing to the weak limit as $n\to + \infty$, we have
\[\|u(t)- N(t,x)\|_{H^1} \leq C t^{- \frac{1}{16}}.\]
Therefore, recall the value of $\alpha$ given in \eqref{eq:alpha}, we have constructed a solution $u(t)$ satisfying the conclusion of Main Theorem.

\subsection*{Acknowledgements}  
This paper has been prepared as a part of my PhD under the supervision of Y.~ Martel. I would like to express my gratitude for his constant support and many helpful discussions.
\\

\end{document}